\documentclass{article}
\usepackage{graphicx}
\usepackage{algpseudocode}
\usepackage{algorithm}
\usepackage{csquotes}

\usepackage{amsmath, amsthm, amssymb}

\newtheorem{proposition}{Proposition}

\newtheorem{lemma}{Lemma}
\newtheorem{remark}{Remark}

\usepackage{caption}
\usepackage{booktabs}
\usepackage[italicdiff]{physics}
\usepackage{subcaption}
\usepackage{xcolor}  

\newcommand{\new}[1]{{\color{black}#1}}

\usepackage{natbib}
\setcitestyle{authoryear,open={(},close={)}, maxcitenames=3} 

\title{How much Data do We Need? \\Sequential Data Collection for\\ Stochastic Programming}
\author{Xin Li, Juergen Branke, Xuan Vinh Doan\\
University of Warwick\\
\texttt{\{xin.li.2, juergen.branke, xuan.doan\}@wbs.ac.uk}}

\date{}

\begin{document}
\maketitle
\begin{abstract}
    Data-driven optimization often requires collecting data to estimate uncertain model parameters before solving the underlying decision problem. In practice, however, data acquisition may incur non-negligible costs, making it critical to determine when to stop additional data collection. In this paper, we study an optimal stopping problem for sequential data collection in stochastic optimization under parameter uncertainty. We propose a benefit-driven stopping framework that balances information gain and sampling cost. We model the unknown distribution parameter within a Bayesian learning framework and update beliefs sequentially as new observations are collected. At each iteration, the decision maker evaluates the expected marginal benefit of additional data relative to the unit sampling cost and determines whether to continue sampling or stop and implement the optimization decision. Based on this framework, we develop several stopping policies. The proposed policies are evaluated through a newsvendor problem with exponentially distributed demand. Numerical experiments compare the policies with fixed-budget and hindsight benchmark strategies. The results show that benefit-driven stopping rules can substantially reduce unnecessary data collection while achieving near-optimal decision performance, demonstrating the effectiveness of adaptive stopping in data-driven optimization.
\end{abstract}

\section{Introduction} \label{sec1:introduction}

Mathematical programming models typically rely on input parameters that are either specified by domain experts or estimated from historical data. In practice, these parameters are rarely known with certainty. When the estimated parameters deviate from their true values, the resulting optimization decisions may be sub-optimal \citep{lam2016advanced}. This challenge is commonly referred to as optimization under input uncertainty \citep{birge2011introduction, shapiro2021lectures} and has received increasing attention in stochastic optimization and data-driven decision making.

A natural approach to mitigating input uncertainty is to collect additional data. More data improve parameter estimates and, consequently, decision quality. However, real-world data collection is often costly and subject to practical constraints \citep{ungredda2022bayesian}. \new{In applications such as healthcare, supply chain management, manufacturing, and service operations, collecting additional data can be costly due to testing expenses, operational disruptions, limited access to data, or time constraints \citep{xu2023decision, fu2010endogenous}.} This creates a fundamental trade-off between the value of information obtained from new data and the cost of collecting it. In this paper, we are interested in finding the trade-off between the benefit and the collection cost of new data. 

Most existing studies in data-driven optimization assume that data are given in a single batch and remain fixed throughout the decision-making process \citep{song2019stochastic, hesong2024introductory}. In contrast, many real-world applications allow data to be collected sequentially over time. \new{For example, demand information may be obtained incrementally through repeated market studies or ongoing observations \citep{besbes2013implications}}. As more data are collected, parameter estimates become more accurate, and the resulting decisions improve. With infinite data, the estimates converge to the true parameters, and thus the corresponding solution converges to the true optimal solution \citep{kleywegt2002sample, kim2014guide, he2024stochastic}. Nevertheless, data collection incurs cost, the decision maker (DM) must balance the benefit of reduced input uncertainty against the cost of further data collection. 

To achieve the balance, we formulate the problem as an optimal stopping problem within a dynamic sequential decision-making framework as follows: a system evolves stochastically from one state to another in discrete time steps. At each decision epoch, the system state is characterized by the current information set (e.g., observed data and posterior beliefs). The DM chooses between two actions: stop data collection and implement the current decision, or continue sampling to acquire additional information \citep{ciocan2022interpretable}. Whether the DM chooses to stop or continue, such a decision is dependent on the current state of the system. Hence, the DM must specify a stopping policy, which prescribes the action to be taken (stop/continue) for each state the system may enter. 
\new{In this paper, we propose stopping policies 
designed to determine a cost-effective stopping point for data collection.}

To evaluate the stopping policies, we consider two key components. We first characterize regret, defined as the optimality gap between the current solution and the true optimal decision \citep{he2024stochastic}. Second, we characterize the benefit of data collection via the expected performance gain from the current solution to the next solution (one-step look-ahead) under the true system. Finally, we assume that the cost of data collection increases linearly with the number of data points. We compare the benefits and the costs of data collection to decide whether to continue data collection or stop. The objective is to determine a stopping policy that minimizes the sum of expected regret and sampling cost. 


The contributions of the paper are as follows: 
\begin{itemize}
    \item We formulate a sequential data acquisition problem as an \emph{optimal stopping problem} within stochastic programming under input uncertainty, explicitly accounting for the trade-off between information gain and sampling cost.
    \item We develop stopping policies to effectively balance the expected value of additional data against its cost, with the objective of minimizing the expected sum of decision regret and sampling cost.
    \item We establish structural and theoretical properties of the proposed policies, showing that they provide principled and adaptive rules for deciding when to continue sampling or terminate data collection. 
    \item \new{We validate the proposed framework through numerical experiments on a classic Newsvendor problem}, demonstrating that the policies achieve near-optimal performance while significantly reducing unnecessary data collection compared to benchmark strategies.
\end{itemize}

The structure of this paper is as follows. We first review the relevant literature in Section \ref{sec2:literature review}. Then, the problem statement is discussed in Section \ref{sec3:problem}. Section \ref{sec4:algorithm} develops stopping policies for data collection. 
In Section \ref{sec6:numericalstudy}, we present the numerical experiment results. The paper finishes in Section \ref{sec7:conclusion} with a conclusion and discussion. 


\section{Literature Review} \label{sec2:literature review}

\subsection{Input uncertainty quantification}
Input uncertainty quantification aims to study how the input randomness propagates to the output. Some methods quantify the \new{output variability}, while other methods construct interval estimates that capture the true performance measure with high confidence or credibility, accounting for both stochastic variability and input uncertainty \citep{lam2016advanced}. Such an interval could be a confidence interval (CI) or a credibility interval (CrI) based on the quantification approach. The methods can be divided into frequentist and Bayesian approaches. The frequentist approaches, such as the \new{direct bootstrap \citep{barton1993bootstrap} and delta methods \citep{cheng1997delta}}, are effective for quantifying parametric input uncertainty but are often computationally intensive. The Bayesian approaches can effectively model both parametric and model uncertainty by incorporating the modeler's prior beliefs about the input distributions, such as Bayesian model averaging \citep{chick2001input}. The challenge of Bayesian methods is the computational effort in estimating the posterior distribution when the posterior distribution is not simple and can only be evaluated approximately via simulation methods \citep{song2014advanced}. \citet{lam2015quantifying} propose a frequentist and Bayesian approach to quantify uncertainty in the optimal output from stochastic optimization and adapt these approaches to quantify the uncertainty in the optimality gap. Input uncertainty can be effectively reduced by collecting more data \citep{hesong2024introductory, ungredda2022bayesian, wang2023input}. \new{Hence, we are interested in how reducing input uncertainty through additional data acquisition can improve the solution quality after optimization.}

\subsection{Stochastic optimization}
Traditional deterministic optimization, which assumes all problem parameters are known and fixed values, provides a powerful but limited framework. It yields a single "best guess" solution that is optimal only for the specifically assumed scenario \citep{birge1997stochastic}. However, the reality can naturally deviate from the assumed scenario, and the deterministically optimal solution may perform poorly, become suboptimal, or even infeasible \citep{avriel1970value}. Stochastic optimization (also called stochastic programming), as one of the approaches for optimization under uncertainty, emerges from the necessity to take uncertainty into account in optimization. It provides a systematic and mathematically rigorous framework for making better, more resilient, and robust decisions under uncertainty \citep{kall1994stochastic}. The origin of stochastic optimization is \citet{dantzig1955linear}, who formulates the first two-stage stochastic linear program and demonstrates how uncertain parameters can be incorporated into linear programming models. Stochastic optimization has been studied for several decades, see \citet{birge2011introduction},\citet{powell2019unified}, and \citet{shapiro2021lectures} for a comprehensive review. It has found successful applications across many sectors, for example, energy \citep{pereira1991multi} and supply chain \citep{sodhi2009modeling}, etc. 

There are several approaches to solve stochastic optimization problems, including exact deterministic reformulations, sampling-based methods, heuristic approaches, and stochastic approximation \citet{robbins1951stochastic}. When the true probability distribution of the uncertain parameters is continuous or involves more than a handful of discrete outcomes, solving the exact deterministic equivalent is computationally impossible. Sampling-based approximations provide an alternative by working with a random subset of scenarios. The Sample Average Approximation (SAA) framework establishes consistency and convergence guarantees for solving stochastic optimization problems using Monte Carlo samples \citep{kleywegt2002sample, kim2014guide}. 
It approximates the expected objective by drawing a finite set of i.i.d.\ scenarios and solving the resulting deterministic optimization model via sample averages \citep{mak1999monte, seljom2021sample}. Hence, the solution quality depends critically on the scenarios. Since SAA relies on representative scenarios, scenario construction plays a central role. Existing work in stochastic programming has therefore focused on two directions: \new{scenario generation and scenario reduction \citep{dupavcova2003scenario, seljom2021sample}. Scenario generation} aims to generate scenarios as representative as possible to ensure a good quality of the solution. There are three common methods to generate scenarios: random sampling, moment matching, and scenario tree construction using distance measures such as probability metrics or Wasserstein distances \citep{seljom2021sample}. An emerging method to generate scenarios is a data-driven method, by using generative models to sample scenarios \citep{choi2025generative}. On the other hand, scenario reduction aims to effectively reduce the computational cost of the problem without affecting the solution quality \citep{bertsimas2023optimization}. The sample size governs how well the empirical distribution approximates the true underlying uncertainty. Larger scenario sets lead to statistically better approximations of the expected objective but increase computational complexity. In this paper, we use the stochastic programming approach to solve the optimization under uncertainty problem. 

\subsection{Data collection in Stochastic Optimization} 
Representative scenarios (data) play a fundamental role in stochastic optimization because model accuracy and solution quality depend directly on the quality of the underlying probability distribution. Classical stochastic optimization works with static in-hand data to approximate the true underlying probability distribution, which is often estimated from historical data \citep{shapiro2021lectures}. However, insufficient and flawed data can lead to biased approximations or suboptimal decisions. Errors in distributional estimation propagate into the optimization solution, making data quality and sample size central to performance.

In most stochastic decision problems, decision makers can collect information that would partially or totally eliminate the inherent uncertainty \citep{hausch1983bounds}. \citet{avriel1970value} provides a foundational contribution in the field of stochastic programming by formally defining and quantifying the value of information in optimization under uncertainty. They look at the question of whether it is worthwhile to buy information about the random variable and how much it would be worthwhile to spend. They argue that decisions made without anticipation of uncertain outcomes may be sub‐optimal, and that acquiring additional information about uncertain events can improve expected outcomes. Before this paper, stochastic programming often lacked an explicit decision-analytic way to analyze whether forecasting is worth the money, whether additional sampling is valuable, or whether learning about distributional inputs is economically beneficial. By quantifying the "best possible improvement with perfect information", \citet{avriel1970value} provides a theoretical upper bound of the expected value of perfect information. Some recent studies show the benefit of streaming data on optimization with input uncertainty. \citet{song2019stochastic} develops an SAA framework to handle a sequence of optimization problems with estimated input parameters based on streaming data. ~\citet{liu2019online} propose a two-layer importance sampling framework for online input uncertainty quantification using streaming data. 

Despite extensive work on stochastic optimization, relatively little attention has been given to the data collection within these problems. In practice, decision makers may acquire additional data, often at a cost, to improve their understanding of the underlying uncertainty \citep{hausch1983bounds, aigner2023data}. As more data become available, the estimated probability distribution more closely approximates the true distribution, thereby yielding solutions that better reflect the true optimal decision. Motivated by this, we work on the idea of actively collecting additional data to enhance solution quality in stochastic optimization in our paper.

\subsection{Optimal Stopping}
Optimal stopping studies the problem of deciding when to terminate an information-acquisition or decision process to optimize an expected objective \citep{ciocan2022interpretable}. It provides a natural framework for sequential decision making under uncertainty when delaying action yields additional information but also incurs a cost. Classical treatments are rooted in sequential analysis and dynamic programming, where the decision maker selects a stopping time to optimize an expected reward. 

In stochastic optimization, optimal stopping has been incorporated into multistage decision problems. For example, \citet{pichler2022risk} study optimal stopping in risk-averse stochastic programming and highlight the role of time consistency under coherent risk measures. An early work of \citet{morton1998stopping} proposes stopping rules based on statistical bounds for optimality gaps in sampling-based stochastic programming. In line with this work, \citet{bayraksan2011sequential} develop a sequential sampling procedure for stochastic programming that increases sample size until the estimated optimality gap is sufficiently small with theoretical guarantees on solution quality and termination. These methods provide principled criteria for stopping, but typically focus on controlling estimation error rather than explicitly modeling the value of collecting additional data. 

Therefore, limited attention has been given to optimal stopping for sequential data collection in stochastic programming. Motivated by this gap, we study the optimal stopping problem of determining when to terminate data collection, explicitly accounting for how additional data reduces input uncertainty and improves solution quality. Recognizing that data acquisition is costly, we develop stopping policies that balance solution improvement against sampling cost.

In summary, the existing literature offers powerful tools for modeling and solving stochastic optimization problems, yet it typically treats the underlying data as fixed and exogenous. The dependence of solution quality on data is well understood, but the strategic decision of whether, when, and how much additional data to collect remains largely unexplored. This gap is especially salient in practical settings where data acquisition is costly and incremental. Building on the literature, our work aims to address this limitation by developing a framework for optimal stopping for sequential data collection in stochastic programming problems. By integrating information acquisition with optimization, we aim to provide a more realistic and effective approach for decision-making under uncertainty. 

\section{Problem formulation} \label{sec3:problem}
We first present the underlying stochastic optimization problem under input uncertainty. We then formulate the sequential data collection process as an optimal stopping problem within a dynamic programming framework. Next, we introduce the performance metric used to evaluate stopping policies. Finally, we motivate the stopping criterion that balances the expected benefit of additional data against the associated sampling cost. 

\subsection{Notation}
Notations

$k$: iteration index, $k \in \{0, 1, \ldots, K\}$;

$\tau$: iteration index at which data collection terminates, $\tau \in \{0, 1, \ldots, K\}$.

$\xi$: random input;

$f(x, \xi)$: reward function;

$\theta$: input distribution parameter to be estimated;  

$\theta^*$: true distribution parameter;

$\mathbb{P}[\theta_k]$: input distribution with distribution parameter $\theta_k$ estimated at the $k$-th iteration;

$\mathbb{P}[\theta^*]$: true input distribution with true distribution parameter $\theta^*$;

$x_k$: optimal solution at the iteration of $k$ under estimated distribution $\mathbb{P}[\theta_k]$; 

$x^*$: optimal solution under true distribution $\mathbb{P}[\theta^*]$; 

$n_k$: the size of the new data; 

$N_k$: the size of current data set at iteration $k$, $N_{k+1} = N_k + n_k$; 

$d_i$: new data points, $d_i \in \mathcal{D}$;

$\mathcal{D}_k$: new data set, $\mathcal{D}_k = \{d_1, d_2, \ldots, d_{n_k}\}$;

$\tilde{ \mathcal D}_{k}$: hypothetical new data set, $\tilde{\mathcal{D}}_{k} = \{\tilde{d}_1,\tilde{d}_2, \ldots, \tilde{d}_{{n}_{k}}\}$; 

$S_k$: state of the system at iteration $k$, representing the cumulative data collected up to that iteration data set at iteration $k$, $S_{k+1} = S_k \cup \mathcal{D}_{k}$; 

$\pi$: stopping policy; 

$c_s$: unit data sampling cost; 

$B$: data sampling budget; 

$\gamma$: confidence level. 

\subsection{Our optimal stopping problem}
Here, we first introduce our fundamental optimization problem. We consider an optimization problem to obtain a solution $x$ with the objective to maximize the expected performance. The performance function $f(x,\xi)$ is assumed to be known and in a closed form. $\xi$ is a random vector, following a distribution with a true distribution parameter $\theta^*$. We estimate it from historical data, denoted as $\theta$. With the estimated $\theta$, the problem is formulated as follows: 
\new{
\begin{align} 
    x^*(\theta) \in \arg \max_{x} \mathbb{E}_{\xi \sim \mathbb P[\theta]}[f(x,\xi)]. \label{formu:max problem}
\end{align}
}

We define the optimal stopping problem as a dynamic sequential decision process. The modeling framework depicts dynamic problems as a sequence of states, decisions, and new information. Let $k \in \{0, 1, 2, \ldots, K\}$ be the iteration of a dynamic decision process. Without loss of generality, we assume that any process has to stop after $K$ iterations, i.e., the decision maker is forced to make a stopping decision at iteration $K$ \citep{oh2016characterizing}. For example, in a survey-based data collection task, $K$ may represent the maximum number of recruitment rounds allowed by a fixed budget or project deadline. We describe the overall process as: at the $k$-th iteration, the modeler estimates the parameter of the distribution, denoted by $\theta_k$. With the estimate, the modeler solves the optimization problem and obtains a solution $x_k$. When $k$ is small, $\theta_k$ may be a poor estimate of the true distribution parameter $\theta^*$, possibly resulting in suboptimal solutions. Assuming additional data are available, one may attempt to collect sufficiently large data so that the estimation of $\theta_k$ gets closer to $\theta^*$. At the $k$-th iteration, a new batch of input data can be collected and then joined with previous data. The new data comes from the true distribution, i.e., $\mathbb{P}[\theta^*]$. With the new data in the $(k+1)$-th iteration, the $k$-th iteration estimator, $\theta_k$, is updated to $\theta_{k+1}$. The optimization problem is solved with the updated estimate. With infinite data, $\theta_k$ converges to $\theta^*$, hence the resulting solution converges to the true optimal solution $x^*$.

\new{Let $S_k$ denote the state of the system at iteration $k$, representing the cumulative data collected up to that iteration, where $N_k$ is the total number of observations available at iteration $k$. $S_k := \{\xi_1,\ldots,\xi_{N_k} \}$.} The distribution parameter estimate $\theta_k$, and the optimal solution $x_k = x^*(\theta_k)$, are functions depending on $S_k$. Initially, we have historical data set $S_0$ from the true distribution $\mathbb{P}[\theta^*]$. It is noted that the historical dataset $S_0$ can be empty. Let $\mathcal{A} = \{\mathrm{stop}, \mathrm{continue}\}$ denote the action space of the problem. We denote the action at the $k$-th iteration as $a_k \in \mathcal{A}$. The decision maker observes the state $S_k$ and decides whether to continue or stop data collection. If the decision maker chooses to stop data collection, they receive a final reward, and the problem is terminated. The final reward is the reward obtained from implementing the final solution selected after data collection stops. However, if the decision maker opts to continue, they receive a cost of continuing and the state transitions from $S_k$ to $S_{k+1}$ as follows. The cost of continuing captures the additional resources needed to collect new data, for example, monetary cost.

First, we update the cumulative data set $S_k$. The dataset collected of size $n_k$ is denoted as $\mathcal{D}_k = \{d_1, \ldots, d_{n_{k}}\}$. Hence, the accumulated data set at the iteration $k+1$ is $S_{k+1} = S_k \cup \mathcal{D}_k$. The size of the accumulated data set $S_{k+1}$ is denoted by $N_{k+1}$, where $N_{k+1} = N_k + n_k$. We note that the new data comes from the true distribution $\mathbb{P}[\theta^*]$. 

Second, we update the distribution parameter estimate $\theta_k$. There are two common estimation approaches: frequentist and Bayesian \citep{lam2016advanced}. The frequentist view assumes that the true distribution parameter is fixed but unknown. The distribution parameter $\theta$ is estimated via maximum likelihood estimation (MLE). On the other hand, the Bayesian approach, which we follow here, assumes that the true distribution parameter is random and follows a distribution. \new{It initially follows a prior distribution $p_0$, which is the initial knowledge and belief about the distribution.} At iteration $k$, using the updated data set $S_{k+1}$, the prior distribution is updated to the posterior distribution $\mathbb{P}[\theta_{k+1}]$, as follows.

\begin{align}
    \mathbb{P}[\theta_{k+1} | S_{k+1}] \propto \mathbb{P}[\theta_k] \sum_{i=1}^{n_k} \mathbb{P}[d_i|\theta_k]. \label{formu:bayes update}
\end{align}

Third, we update the optimization solution from $x_k$ to $x_{k+1}$ by solving the optimization problem via stochastic programming. With the updated $\theta_{k+1}$ after data collection, we re-solve the problem and get the updated solution $x_{k+1} =x^*(\theta_{k+1}) \in \arg\max_x \mathbb{E}_{\xi \sim \mathbb P[\theta_{k+1}]}[f(x,\xi)]$.

The focus of our work is to design a stopping policy $\pi$ to determine the \emph{stopping time} $\tau$. The stopping time $\tau$ is a random variable that takes values in $\{0, 1, 2, \ldots, K\}$. With the assumption of $K$, the problem is a finite horizon problem with $K$ iterations, starting at iteration $k = 0$. The stopping policy $\pi$ is a mapping from states to actions, namely $a_k = \pi(S_k)$. At the stopping time $\tau_\pi$, the decision maker terminates data collection. 

\subsection{Stopping policy performance metric}
To evaluate the performance of a stopping policy, we consider two components: solution quality and data sampling cost. First, the solution quality is evaluated via regret. Regret, also referred to as the optimality gap or opportunity cost, is the solution performance loss of not knowing the true distribution. \new{The regret $RE(\theta^*, S_k)$ is determined by the difference in true performance between $x_k$ and the true optimal solution $x^*$: 
\begin{align}
   RE(\theta^*, S_k) = \mathbb{E}_{\xi \sim \mathbb{P}[\theta^*]} [f(x^*, \xi) - f(x_k, \xi)], \label{formu:regret}
\end{align}
}
where $x^*$ is obtained by solving the stochastic programming problem under the true distribution, i.e., $x^* = x^*(\theta^*) \in \arg \max_x \mathbb{E}_{\xi \sim \mathbb{P}[\theta^*]} [f(x, \xi)]$. We note that the true distribution parameter $\theta^*$ is unknown in practice, but it is used for retrospective performance evaluation. 

Second, we assume that the data sampling cost is linear with the number of data points collected. 
By the stopping point $\tau_\pi$ under any policy $\pi$, the number of data samples is $N_{\tau_\pi}$, where $N_{\tau_\pi} = \sum_{k=1}^{\tau_\pi} n_k$. Hence, the cumulative sampling cost $SC_{\tau_\pi}$ is 

\begin{align}
    SC_{\tau_\pi} = c_s N_{\tau_\pi} = c_s \sum_{k=1}^{\tau_\pi} n_k. \label{formu:sampling cost}
\end{align}
where $c_s$ is the unit sampling cost. 

Overall, the performance of a stopping policy is evaluated via the sum of the regret of the solution at the stopping point $\tau_\pi$ and the cumulative sampling cost up to the stopping point $\tau_\pi$, referred to as $RESC_{\tau_\pi}$, is 

\begin{align}
    RESC_{\tau_\pi} = RE(\theta^*, S_{\tau_\pi}) + SC_{\tau_\pi} = \mathbb{E}_{\xi \sim \mathbb{P}[\theta^*]} [f(x^*, \xi) - f(x_{\tau_\pi}, \xi)] + c_s \sum_{k=1}^{\tau_\pi} n_k. \label{formu:performance metric}
\end{align}

\new{The objective of the optimal stopping problem is to find a stopping policy $\pi$ that determines when to stop data collection, yielding the stopping time $\tau_\pi$, such that $RESC_{\tau_\pi}$ can be minimized in expectation. The optimal stopping problem can be formulated as }

\begin{align}
    & \min_{\pi} \int RESC_{\tau_\pi(\omega)} P(\omega) d\omega \nonumber \\
    = & \min_{\pi} \int \left( \mathbb{E}_{\xi \sim \mathbb{P}[\theta^*]} [f(x^*, \xi) - f(x_{\tau_\pi(\omega)}, \xi)] + c_s \sum_{k=1}^{\tau_\pi(\omega)} n_k \right) P(\omega) d\omega. 
\end{align}
\new{where $\omega = \{S_1, S_2, S_3, \ldots, S_{\tau_\pi}\}$ is a sample path and $P(\omega)$ is the probability distribution of the sample path $\omega$. }

\new{We use an example of $\tau_\pi (\tau_\pi \leq K)$ to illustrate the whole process of data collection until the terminal iteration $\tau$ under stopping policy $\pi$. The trajectory of the process up to the iteration $\tau_\pi$ is illustrated in Table~\ref{tab:illustrative example} below.}

\begin{table}[htb]
    \centering
    \caption{Illustrative example when stopping at $\tau_\pi$}
    \label{tab:illustrative example}
    \begin{tabular}{|l c c c c c c c |c|} 
         \hline
         Iteration $k$ & 0 & 1 & \ldots & $k$ & $k+1$ & \ldots & $\tau_\pi$ \\ 
         Action & continue  & continue  & \ldots & continue  & continue  & \ldots & stop   \\
         Cumulative data & $S_0$  & $S_1$ & \ldots & $S_k$ & $S_{k+1}$ & \ldots & $S_{\tau_\pi}$ \\ 
         New data       & $\mathcal{D}_0$  & $\mathcal D_1$ & \ldots & $\mathcal{D}_{k}$ & $\mathcal{D}_{k+1}$ & \ldots & \\ 
         Distribution estimate & $\theta_0$  & $\theta_1$ & \ldots & $\theta_k$ & $\theta_{k+1}$ & \ldots & $\theta_{\tau_\pi}$ \\ 
         Stochastic solution & $x_0$  & $x_1$  & \ldots & $x_k$ & $x_{k+1}$ & \ldots & $x_{\tau_\pi}$   \\ \hline
    \end{tabular}
\end{table}

\subsection{Stopping policy motivation}
Here, we introduce the motivation for our proposed optimal stopping policy. We first introduce the \textit{Cost of actions.} We introduce the cost-to-stop and cost-to-continue formulations under a dynamic programming framework. The cost-to-stop, also referred to as the terminal value function in a minimization problem, is denoted by $C_{\mathrm{stop}}(S_k)$. \new{It is defined as the regret incurred by terminating the data collection process at iteration (k), i.e., $C_{\mathrm{stop}} = \mathrm{RE}(\theta^*, S_k)$}. 
It represents the expected performance loss incurred by stopping data collection at state $S_k$ and adopting the current solution $x_k$. 

The cost-to-continue is denoted by $C_{\mathrm{cont}}(S_k)$. Continuing the data collection incurs an immediate sampling cost, $c_s n_k$. Hence, we define the cost-to-continue as 
\begin{align}
    C_{\mathrm{cont}}(S_k) = c_s  n_k.
\end{align}
In addition, continuing leads to an updated solution based on the augmented dataset $S_{k+1}$ and the regret when one decides to stop in the future.

We define the value function \(V_k(S_k)\) as the minimum expected performance loss at iteration $k$, conditional on the current state \(S_k\). At the terminal period \(K\), the decision maker is forced to stop. Therefore, the boundary condition is
\begin{align}
    V_K(S_K) = C_{\mathrm{stop}}(S_K) = \mathbb{E}_{\xi \sim \mathbb{P}[\theta^*]}  \left[f(x^*,\xi)-f(x_K,\xi)\right].
\end{align}

Let $d =(d_1,\ldots,d_{n_k})$ denote the batch of new observations to be collected at iteration $k$. Assuming the observations are independently generated from the true distribution \(\mathbb{P}[\theta^*]\), we have $d \sim \prod_{i=1}^{n_k}\mathbb{P}[\theta^*]$. For any period \(k < K\), the decision maker chooses between stopping immediately and continuing data collection. Thus, we define the Bellman equation as
\begin{align}
     V_k(S_k) = \min \Bigg\{ C_{\mathrm{stop}}(S_k), \; C_{\mathrm{cont}}(S_k) + \mathbb{E}_{d \sim \prod_{i=1}^{n_k}\mathbb{P}[\theta^*]} \big[ V_{k+1}(S_{k+1}) \big] \Bigg\}. \label{eq:bellman}
\end{align}

The optimal decision at period \(k\) is to stop if
\begin{align}
    C_{\mathrm{stop}}(S_k) \leq C_{\mathrm{cont}}(S_k) + \mathbb{E}_{d \sim \prod_{i=1}^{n_k}\mathbb{P}[\theta^*]} \big[ V_{k+1}(S_{k+1}) \big], 
\end{align} 
and to continue otherwise.

In principle, the optimal policy can be obtained by solving the Bellman equation backward from \(k=K\) to \(k=1\). However, computing the expected loss \(V_{k+1}(S_{k+1})\) of iteration $k+1$ exactly is generally difficult, since it requires evaluating the remaining finite-horizon dynamic program. Hence, we adopt a one-step look-ahead approximation. From Formulation (\ref{eq:bellman}), we can see that $V_{k+1}(S_{k+1})$ has an upper bound of \( C_{stop}(S_{k+1})\), i.e., \(V_{k+1}(S_{k+1}) \leq C_{stop}(S_{k+1})\). Therefore, we approximate the future value function \(V_{k+1}(S_{k+1})\) by the immediate stopping loss at the next iteration:
\begin{align}
    V_{k+1}(S_{k+1}) \approx C_{\mathrm{stop}}(S_{k+1}) = \mathbb{E}_{\xi \sim \mathbb{P}[\theta^*]} \left[f(x^*,\xi)-f(x_{k+1},\xi)\right].
\end{align}

Under this approximation, the decision maker stops at period \(k\) if the immediate stopping loss is no larger than the approximate continuation value:
\begin{align}
    C_{\mathrm{stop}}(S_k)  \leq C_{\mathrm{cont}}(S_k) + \mathbb{E}_{d \sim \prod_{i=1}^{n_k}\mathbb{P}[\theta^*]} \left[ C_{\mathrm{stop}}(S_{k+1}) \right]. 
\end{align}

Therefore, the stopping time $\tau$ induced by a heuristic stopping policy \(\pi\) is defined as the first time in which the above condition holds:
\begin{align}
    \tau_{\pi}
    = \inf \left\{ k \in \{1,\ldots,K-1\}: C_{\mathrm{stop}}(S_k) \leq C_{\mathrm{cont}}(S_k) + \mathbb{E}_{d \sim \prod_{i=1}^{n_k}\mathbb{P}[\theta^*]} \left[ C_{\mathrm{stop}}(S_{k+1})  \right] \right\}. 
\end{align}
If the stopping condition is not satisfied for any \(k < K\), then the decision maker stops at the terminal period \(K\).

\new{ When $C_{\mathrm{cont}}(S_k)=c_s n_k$, we allow the number of newly collected data points at iteration $k$ to be general. The stopping condition can then be written as

\begin{align}
    \mathbb{E}_{\xi \sim \mathbb{P}[\theta^*]}
    \big[ f(x^*, \xi) - f(x_k(\theta_k), \xi) \big] - \mathbb{E}_{d \sim \prod_{i=1}^{n_k}\mathbb{P}[\theta^*]} 
    \Big[ \mathbb{E}_{\xi \sim \mathbb{P}[\theta^*]} \big[f(x^*, \xi) - f(x_{k+1}(\theta_{k+1}(d)), \xi) \big] \Big] \leq c_s n_k. \label{eq:bellman optimal}
\end{align}

Since $\mathbb{E}_{\xi \sim \mathbb{P}[\theta^*]} [ f(x^*, \xi)]$ is a constant and does not depend on the newly collected data $d$, the above condition can be equivalently reformulated as 

\begin{align}
    \mathbb{E}_{d \sim \prod_{i=1}^{n_k}\mathbb{P}[\theta^*]}
    \Big[ \mathbb{E}_{\xi \sim \mathbb{P}[\theta^*]} \big[ f(x_{k+1}(\theta_{k+1}(d)), \xi) - f(x_k(\theta_k), \xi) \big] \Big] \leq c_s n_k. \label{eq:bellman_general_final}
\end{align}

Inequality~\eqref{eq:bellman_general_final} shows that the decision maker should stop collecting data when the expected improvement in solution quality from collecting a batch of $n_k$ new observations is no larger than the corresponding data collection cost $c_s n_k$. 
}





Equation (\ref{eq:bellman optimal}) is the cornerstone of how we design stopping policies. It justifies our stopping mechanism based on a comparison between the expected benefit of collecting additional data and the associated sampling cost. The left-hand side of \eqref{eq:bellman optimal} represents the expected real benefit (improvement) in solution quality, while the right-hand side corresponds to the sampling cost of collecting additional data. The expected real benefit, therefore, serves as a critical stopping criterion for data collection policies and will be further discussed in Section~\ref{sec4:algorithm}.

\new{\begin{remark}[Conservativeness of the heuristic stopping rule] 
The proposed heuristic stopping rule provides a tractable and conservative approximation to the Bellman-optimal stopping rule. Since the exact future value function is difficult to compute, we approximate it with the immediate stopping loss in the next iteration. This approximation makes the continuation option appear weakly more costly than it is under the true Bellman recursion, and therefore tends to favor stopping earlier. In this sense, the heuristic stopping time is expected to be no later than the Bellman-optimal stopping time along the same sample path. 
\end{remark} }


\section{Policy development} \label{sec4:algorithm}
\new{In this section, we develop stopping policies for sequential data collection. Since the true data-generating distribution is unknown, the real benefit of collecting additional data cannot be evaluated directly. Instead, the decision maker updates the belief about $\theta^*$ using the available data. Under this belief, the benefit of data collection becomes a random quantity induced by the posterior distribution of $\theta$. This motivates us design stopping policies that use summary measures of the random benefit, such as its posterior mean or credible interval, as stopping indicators.
We first introduce the benefit measure in Section~\ref{subsec:benefit measure}, and then present several policy designs based on different approximations of the unknown real benefit.}

\subsection{Policy stopping measure} \label{subsec:benefit measure}


When the true distribution parameter $\theta^*$ is known, we define the \emph{real benefit} of data collection as the actual performance improvement obtained after collecting additional data. Suppose that the newly collected data at stage $k+1$ has size $n_{k+1}$. \new{We assume that new data are generated from the same distribution as the historical observations of $\xi$, namely, $\mathbb{P}[{\theta^*}]$.} 
\new{The batch of new observations $\mathcal{D}_{k} = \{d_1,\ldots,d_{n_k}\}$ at iteration $k$ are independently generated from the true distribution \(\mathbb{P}[\theta^*]\), i.e., $d \sim \prod_{i=1}^{n_k}\mathbb{P}[\theta^*]$. $d = (d_1,\ldots,d_{n_k})$. }
The cumulative data set is then updated as $S_{k+1} = S_k \cup \mathcal{D}_{k+1}$. Let $x_k$ and $x_{k+1}(d)$ denote the stochastic solutions obtained based on state $S_k$ and $S_{k+1}$, respectively. 
The real benefit of data collection, denoted by $RB(\theta^*, S_k)$, is defined as the reduction in regret achieved after incorporating the newly observed data, i.e., 

\begin{align}
      RB(\theta^*, S_k) & = \mathbb{E}_{d \sim \prod_{i=1}^{n_k}\mathbb{P}[\theta^*]} [\mathrm{RE}(\theta^*, S_k) - \mathrm{RE}(\theta^*, S_{k+1})] \nonumber \\
      & = \mathbb{E}_{d \sim \prod_{i=1}^{n_k}\mathbb{P}[\theta^*]}[\mathbb{E}_{\xi \sim \mathbb{P}[\theta^*]} [f(x_{k+1}(d),\xi) - f(x_{k},\xi)]]. \label{formu:real benefit}
\end{align}
where $RE(\theta^*, S_k)$ and $RE(\theta^*, S_{k+1})$ represent the regrets of $x_k$ and $x_{k+1}(d)$ under the true distribution parameter $\theta^*$, respectively.

However, the true distribution parameter $\theta^*$ is not known in reality. Hence, we introduce the \emph{estimated benefit} or gain from hypothetically collecting additional data. At iteration $k$, we first estimate the distribution parameter $\theta_k$ given the available data, resulting the estimated distribution $\mathbb{P}[\theta_k]$. Then we hypothetically sample $n_k$ new data from the estimated distribution $\mathbb{P}[\theta_k]$. The new data set is $\tilde{\mathcal{D}}_k = \{\tilde{d}_{n_k}\}$, where each data $\tilde{d}_i$ is sampled from the estimated distribution $\mathbb{P}[\theta_k]$. The cumulative data size is updated to $N_{k+1} = N_k + n_k$. With the dataset $\tilde{S}_{k+1} = S_k \cup \tilde{\mathcal{D}}_k$, the updated estimated distribution parameter and distribution are $\tilde{\theta}_{k+1}$ and ${\mathbb{P}[\tilde{\theta}_{k+1}]}$. The performance of the hypothesized solution $\tilde{x}_{k+1}$ is $\mathbb{E}_{\xi \sim \mathbb{P}[\tilde{\theta}_{k+1}]}[f(\tilde{x}_{k+1},\xi)]$. To evaluate how good the solution $x_k$ performs under the updated estimated distribution $\mathbb{P}[\tilde{\theta}_{k+1}]$, we have $\mathbb{E}_{\xi \sim \mathbb{P}[\tilde{\theta}_{k+1}]}[f(x_k,\xi)]$. \new{Built on the performance difference, the estimated real benefit is $RB(\theta_k, S_k)$, by replacing $\theta^*$ with $\theta_k$ in Formulation (\ref{formu:real benefit}). 

}

We adopt this estimated benefit as the \new{policy stopping indicator} since the true underlying distribution is unknown. The estimated benefit quantifies the potential value of additional data collection based on current information. It therefore serves as a forward-looking measure to help the decision maker determine whether initiating further data collection is worthwhile. 

A Bayesian approach is adopted to estimate the input distribution parameter $\theta$. Under this framework, the input parameter $\theta$ is unknown and treated as a random variable. Given a prior distribution $p_0$, the posterior distribution $\mathbb{P}[\theta \mid S_k]$ is updated after collecting $N_k$ additional data points. We focus on estimating two quantities: the distribution parameter $\theta$ and the benefit measure. Based on these estimates, we develop the following four stopping policies. For parameter estimation, we propose two policies: the Parameter Mean policy (PM) policy and the Parameter Credibility Interval policy (PI) policy, presented in Sections \ref{subsec:parameter mean} and \ref{subsec:parameter interval}, respectively. For the benefit measure, we propose two additional policies: the Measure Mean policy (MM) policy and the Measure Credibility Interval policy (MI) policy, described in Section \ref{subsec:measure mean} and \ref{subsec:measure interval}, respectively.

\subsection{Parameter Related Policies} 
Suppose we have a prior belief about the distribution of $\theta$: $\theta \sim \Gamma(\alpha_0, \beta_0)$, where $\alpha_0$ and $\beta_0$ are the shape and rate parameters of the Gamma distribution, respectively. After having observed $N_k$ i.i.d.\ samples at iteration $k$, we update the prior distribution to the posterior distribution 
$\theta \mid \xi_{1:N_k} \sim \Gamma(\alpha_0 + N_k, \beta_0 + \sum_{i=1}^{N_k} \xi_i)$.

We adopt a conjugate prior so that the posterior distribution can be analytically updated. The prior parameters are set to be noninformative, 
i.e., $\alpha_0 = 0$ and $\beta_0 = 0$. After collecting data up to iteration $k$, the posterior parameters are updated to $\alpha_k$ and $\beta_k$, so that $\theta \sim \Gamma(\alpha_k, \beta_k)$. The probability density function of the Gamma distribution is given by
$f(\theta) = \frac{\beta^{\alpha}}{\Gamma(\alpha)} \theta^{\alpha-1} e^{-\beta \theta}$ with $\theta > 0.$. 

Based on the posterior distribution, we estimate the input parameter in two ways: the posterior mean and the posterior credibility interval. These lead to two stopping policies: the Parameter Mean policy and the Parameter Credibility Interval policy. The two policies are described in the following subsections.

\subsubsection{Parameter mean policy} \label{subsec:parameter mean}
\textbf{Parameter Mean Policy (PM)}. We estimate the distribution parameter $\theta$ by its posterior mean at iteration $k$, denoted by $\bar{\theta}_k$. Given the posterior distribution $\theta \sim \Gamma(\alpha_k,\beta_k)$, the posterior mean is $\bar{\theta}_k = \frac{\alpha_k}{\beta_k}$.

We then substitute $\bar{\theta}_k$ into the benefit function in Formulation~(\ref{formu:real benefit}), yielding $RB(\bar{\theta}_k, S_k)$. To determine whether to stop data collection, we compare $RB(\bar{\theta}_k, S_k)$ with the sampling cost $c_s n_K$. If $RB(\bar{\theta}_k, S_k) < c_s n_k$, data collection stops; otherwise, it continues. The corresponding procedure is summarized in Algorithm~\ref{alg:parameter mean}.

\begin{algorithm}[htbp]
\caption{Parameter Mean (PM) Stopping Policy: Sequential data collection when estimating the posterior mean of the random true parameter.}
\label{alg:parameter mean}
\begin{algorithmic}[1]
    \State \textbf{Initialize:} Set iteration index $k \gets 0$, cost $b \gets 0$, $\alpha_0 \gets 0,\beta_0 \gets 0$
    \State Collect initial dataset $S_0$ from the true distribution, update the posterior distribution $\theta \sim \Gamma(\alpha_0,\beta_0)$, compute the distribution estimator $\bar{\theta}_0 \gets \mathbb{E}[\theta]$, obtain the initial decision $x_0$, compute the estimated benefit of the mean $RB(\bar{\theta}_0, S_0)$. 
     \While{$RB(\bar{\theta}_k, S_k) \geq c_s n_k$}
        \State Collect data set $\mathcal{D}_k$ of size $n_k$ from true distribution
        \State Update dataset: $S_{k+1} \gets S_k \cup \mathcal{D}_k$
        \State Update the posterior distribution $\theta \sim \Gamma(\alpha_{k+1}, \beta_{k+1})$
        \State Update the distribution estimator $\bar{\theta}_{k+1} \gets \frac{\alpha_{k+1}}{\beta_{k+1}}$
        \State Compute the decision $x_{k+1}$ under $\bar{\theta}_{k+1}$
        \State Compute the estimated benefit: $RB(\bar{\theta}_k, S_k) \gets RB(\bar{\theta}_{k+1}, S_{k+1})$
        \State Accumulate cost: $b \gets b + c_s n_k$
        \State Increment iteration index: $k \gets k + 1$
    \EndWhile
    \State \textbf{Return:} final recommended decision $x_k$ 
\end{algorithmic}
\end{algorithm}

\subsubsection{Parameter credibility interval policy} \label{subsec:parameter interval}
\textbf{Parameter Credibility Interval Policy (PI)}. 
Using the Bayesian posterior distribution, we construct a $100(1-\gamma)\%$ credibility interval for $\theta$. Since $2\beta_k\theta$ follows a chi-square distribution
with $2\alpha_k$ degrees of freedom, the credibility interval is shown in Formulation (\ref{Parameter CrI}). 
\begin{align}
    P\!\left(\frac{\chi^2_{2\alpha_k,\frac{\gamma}{2}}}{2\beta_k} \le \theta \le \frac{\chi^2_{2\alpha_k,1-\frac{\gamma}{2}}}{2\beta_k} \right)= 1 - \gamma .  \label{Parameter CrI}
\end{align}

We denote the interval at iteration $k$ as $[l_k,u_k]$, where
\[l_k=\frac{\chi^2_{2\alpha_k,\frac{\gamma}{2}}}{2\beta_k}, \qquad u_k=\frac{\chi^2_{2\alpha_k,1-\frac{\gamma}{2}}}{2\beta_k}.\]

This interval implies that, given the prior distribution and the observed data, the true parameter $\theta$ lies within $[l_k, u_k]$ with probability $100(1-\gamma)\%$.

Based on this interval, we bound the estimated benefit by varying
$\theta$ within $[l_k,u_k]$. The upper bound of the estimated benefit is
\begin{align}
    URB_k = \max_{\theta \in [l_k, u_k]} RB(\theta, S_k) \label{eq:UpperEstimatedBenefit}
\end{align}

Assuming a conservative decision maker, the stopping rule compares $URB_k$ in Formulation (\ref{eq:UpperEstimatedBenefit}) with the sampling cost $c_s n_k$. Data collection continues if $URB_k \ge c_s n_k$ and stops otherwise. Let $\tau$ denote the stopping iteration, defined as the first $k$ such that $URB_k \leq c_s n_k $. The procedure is summarized in Algorithm~\ref{alg:parameter interval}.


\begin{algorithm}[htbp]
\caption{Parameter Credibility Interval (PI) Stopping Policy: Sequential data collection when estimating the credibility interval of the random true parameter.} 
\label{alg:parameter interval}
\begin{algorithmic}[1]
    \State \textbf{Initialize:} Set iteration index $k \gets 0$, cost $b \gets 0$. 
    \State Collect initial dataset $S_0$, compute $\theta_0 \gets \frac{\alpha_0}{\beta_0}$, obtain the initial decision $x_0$ under $\theta_0$, construct \new{$100(1-\gamma)\%$} credibility interval for $\theta$: $[l_0, u_0]$, and compute expected benefit upper bound $URB_0$ 
    \While{$URB_k \geq c_s n_k$}
        \State Collect data set $\mathcal{D}_k$ of size $n_k$ from true distribution
        \State Update dataset: $S_{k+1} \gets S_k \cup \mathcal{D}_k$
        \State Update posterior distribution $\alpha_{k+1}$ and $\beta_{k+1}$, get $\theta_{k+1} \gets \frac{\alpha_{k+1}}{\beta_{k+1}})$
        \State Compute decision $x_{k+1}$ under $\theta_{k+1}$
        \State Accumulate cost: $b \gets b + c_s n_k$
        \State Recompute the $100(1-\gamma)\%$ credibility interval $[l_{k+1}, u_{k+1}]$ and update the upper bound of the expected estimated benefit bound $URB_{k+1}$
        \State Increment iteration index: $k \gets k + 1$
    \EndWhile
    \State \textbf{Return:} final recommended decision $x_k$ 
\end{algorithmic}
\end{algorithm}

\subsection{Measure Related Policies}
In this subsection, we focus on the estimated benefit function, $RB(\theta_k, S_k)$ at iteration $k$. 
Since the distribution parameter $\theta$ follows a posterior distribution, $RB(\theta_k, S_k)$ itself becomes a random variable. Given the posterior distribution $\theta \sim \Gamma(\alpha_k,\beta_k)$, we estimate the benefit in two ways: the posterior mean of the benefit and a credibility interval of the benefit. Based on these two approaches, we develop two policies: the Measure Mean policy and the Measure Credibility Interval policy.

\subsubsection{Measure mean policy} \label{subsec:measure mean}
\textbf{Measure Mean Policy (MM)}. In this policy, we use a Bayesian framework to update our belief about the true parameter $\theta$. The posterior mean of the estimated benefit is used as the stopping criterion.  
$\alpha_0=0$ and $\beta_0=0$.

Using the posterior distribution for $\theta$, the posterior mean of the estimated benefit, $MRB_k$, is given by

\begin{align}
    & MRB_k = \mathbb{E}_{\theta \sim \Gamma(\alpha_k, \beta_k)} [\mathrm{RB}(\theta, S_k)]. \label{eq:Mean_EB}
\end{align}

\new{$MRB_k$ is computed by integrating the expected benefit $\mathrm{RB}(\theta, S_k)$ over the posterior distribution $\Gamma(\alpha_k,\beta_k)$, thereby averaging the benefit across all plausible parameter values under the current posterior. In our test problem in Section \ref{sec6:numericalstudy}, this posterior expectation admits a closed-form analytical expression.}

To determine whether to continue data collection, we compare $MRB_k$ with the sampling cost $c_s n_k$. If $MRB_k < c_s n_k $, data collection stops; otherwise, it continues. The procedure is summarized in Algorithm~\ref{alg:measure mean}.

\begin{algorithm}[htbp]
\caption{Measure Mean (MM) Stopping Policy: Sequential data collection when estimating the benefit measure mean with a random true parameter. } 
\label{alg:measure mean}
\begin{algorithmic}[1]
    \State \textbf{Initialize:} Set iteration index $k \gets 0$, cost $b \gets 0$, $\alpha_0 \gets 0,\beta_0 \gets 0$
    \State Collect initial dataset $S_0$, update the posterior distribution $\theta \sim \Gamma(\alpha_0,\beta_0)$, compute the distribution estimator $\theta_0 \gets \mathbb{E}[\theta]$, obtain the initial decision $x_0$, compute the mean of the estimated benefit $MRB_0$. 
    \While{ $MRB_k \geq c_s n_k$ }
        \State Collect data set $\mathcal{D}_k$ of size $n_k$ from true distribution
        \State Update dataset: $S_{k+1} \gets S_k \cup \mathcal{D}_k$
        \State Update the posterior distribution $\theta \sim \Gamma(\alpha_{k+1}, \beta_{k+1})$
        \new{\State Analytically Compute the mean estimated benefit: $MRB_k \gets MRB_{k+1}$}
        \State Accumulate cost: $b \gets b + c_s n_k$
        \State Increment iteration index: $k \gets k + 1$
    \EndWhile
    \State \textbf{Return:} final recommended decision $x_k$
\end{algorithmic}
\end{algorithm}

\subsubsection{Measure credibility interval policy} \label{subsec:measure interval}
\textbf{Measure Credibility Interval Policy (MI)}. Instead of using the mean of the estimated benefit, this policy constructs a credibility interval for the benefit $RB(\theta, S_k)$ and uses its upper bound as the stopping criterion.

Since $\theta \sim \Gamma(\alpha_k,\beta_k)$, the benefit $RB(\theta, S_k)$ becomes a random variable induced by the posterior distribution of $\theta$. To estimate the credibility interval of $RB(\theta, S_k)$, we employ Monte Carlo sampling. Specifically, we draw $M$ independent samples $\theta^{(1)},\dots,\theta^{(M)}$ from $\Gamma(\alpha_k,\beta_k)$. For each sample, we compute the corresponding benefit: $ RB^{(i)} = RB(\theta^{(i)}), i=1,\dots,M$. The resulting benefit values are sorted to obtain the $100 (1-\gamma) \%$ credibility interval of the benefit, whose lower and upper bounds are denoted by $LRB_k^C$ and $URB_k^C$, respectively. 

Assuming a conservative decision maker, we use the upper bound
$URB_k^C$ as the stopping criterion. If $URB_k^C < c_s n_k$, data collection stops; otherwise, it continues. The procedure is summarized in Algorithm~\ref{alg:measure interval}.

\begin{algorithm}[htbp]
\caption{Measure Credibility Interval (MI) Stopping Policy: Sequential data collection when estimating the credibility interval of the benefit measure with a random true parameter.} 
\label{alg:measure interval}
\begin{algorithmic}[1]
    \State \textbf{Initialize:} Set iteration index $k \gets 0$, cumulative cost $b \gets 0$, prior hyperparameter $(\alpha_0,\beta_0)$
    \State Collect initial dataset $S_0$ and update the posterior distribution $\theta \sim \Gamma(\alpha_0,\beta_0)$
    \State Obtain initial decision $x_0$ 
    \While{$URB_k^c \geq c_s n_k$}
        \State Collect data set $\mathcal{D}_k$ of size $n_k$ from true distribution
        \State Update dataset: $S_{k+1} \gets S_k \cup \mathcal{D}_k$
        \State Update posterior parameters $\alpha_{k+1}, \beta_{k+1}$
        \State Draw a sample of size $M$ of $\theta$ from $\Gamma(\alpha_{k+1}, \beta_{k+1})$
        \State Compute benefits $RB_{k+1}$ and update the $100C\%$ determine the credibility bound of benefit $URB_{k+1}^c$
        \State Compute $x_{k+1}$
        \State Accumulate sampling cost: $b \gets b + c_s n_k$
        \State Increment iteration index: $k \gets k + 1$
    \EndWhile
    \State \textbf{Return:} final recommended decision $x_k$
\end{algorithmic}
\end{algorithm}

\paragraph{We also propose a benchmark policy for comparison: the fixed-budget (FB) policy. 
Under the FB policy, the data collection budget is specified in advance, as described in Section~\ref{subsec:fixed}. 
}

\subsection{Fixed budget policy} \label{subsec:fixed}
\textbf{Fixed Budget Policy}. We use the fixed budget policy as a benchmark. Under this policy, a fixed budget of $B$ is allocated for data collection. Data collection continues until the budget is exhausted. The corresponding procedure is presented in Algorithm~\ref{alg:Fixed}.

\begin{algorithm}[htbp]
\caption{Fixed Budget (FB) Stopping Policy: Sequential data collection under sampling budget constraint.} 
\label{alg:Fixed}
\begin{algorithmic}[1]
    \State \textbf{Initialize:} Set budget spent $b \gets 0$ and iteration index $k \gets 0$. 
    \State Sample initial dataset $S_0$ from the true distribution, compute the distribution estimate $\theta_0$, obtain initial solution $x_0$, and compute the regret $RE_0$
    \While{$b < B$}
        \State Collect new data set $\mathcal{D}_k $ of size $n_k$ 
        from the true distribution 
        \State Update dataset: $S_{k+1} \gets S_k \cup \mathcal{D}_k$
        \State Update distribution parameter estimate: $\theta_{k+1}$ from dataset $S_{k+1}$ 
        \State Compute updated solution: $x_{k+1}$
        \State Update budget: $b \gets b + c_s n_k $
        \State Increment iteration index: $k \gets k + 1$
    \EndWhile
    \State \textbf{Return:} final recommended decision $x_k$
\end{algorithmic}
\end{algorithm}

\subsubsection{Two stopping bounds}
Here, we introduce two lower bounds for the RESC defined in Formulation (\ref{formu:performance metric}): hindsight best stopping bound and hindsight best fixed stopping bound. 

\textbf{Hindsight best stopping bound.} 
The hindsight best stopping bound represents the ideal stopping decision that could have been made if full knowledge of the realized trajectory in each replication would be available. Under a fixed budget policy, data collection continues until the budget is exhausted. Then, ex post, we identify the iteration at which the RESC attains its minimum along the realized path. This iteration corresponds to the point at which data collection should have stopped in hindsight. For each replication, we record the minimum RESC value over all iterations. The hindsight best stopping bound is then computed as the average of these minimum RESC values across all replications. In other words, we first take the minimum within each replication and then average these minima over all runs. 

\textbf{Hindsight best fixed stopping bound.} 
In contrast to the hindsight best bound, which allows the stopping time to vary across replications, the hindsight best fixed stopping bound imposes a common stopping iteration for all replications. Specifically, at each iteration, we compute the average RESC across replications. The hindsight best fixed stopping bound is defined as the minimum of these average RESC values over all iterations. This bound, therefore, identifies a single iteration that would have performed best on average, rather than optimizing the stopping point individually for each replication. 


\section{Numerical study} \label{sec6:numericalstudy}
In this section, we use the classic Newsvendor problem to demonstrate the performance of the proposed stopping policies. 
The structure of this section is organized as follows. Section~\ref{subsec:study setting} describes the experimental settings of the Newsvendor problem.
Section~\ref{subsec:analytical formulation} presents the analytical formulations used in the experiments. We then report the numerical results in three parts:
(i) a comparison of the stopping policies in Section~\ref{subsec:stopping}, (ii) experiments under different true parameters in Section~\ref{subsec:unknown true test}, and (iii) the impact of the unit sampling cost in Section~\ref{subsec:sampling cost test}. 

\subsection{Numerical study setting} \label{subsec:study setting}
The Newsvendor problem determines the optimal order quantity that maximizes the expected profit under uncertain demand. We consider a single-product setting without salvage value or shortage penalty. The decision variable $x$ denotes the number of newspapers ordered, and $\xi$ represents the random demand. The profit function $f(x,\xi)$ is given by $f(x,\xi) = r \min(x,\xi) - cx$, where $r$ is the selling price and $c$ is the unit purchasing cost, with $r>c$. In the numerical experiments, we set $r=10$ and $c=3$. The decision maker aims to determine the optimal order quantity $x$ based on historical demand observations to maximize the expected profit $\mathbb{E}_{\xi \sim \mathbb{P}[\lambda]}[f(x,\xi)]$. We assume that the uncertain demand follows an exponential distribution with an unknown parameter $\lambda$, i.e., $\xi \sim \mathrm{Exp}(\lambda)$. The problem is formulated as follows:

\begin{align}
    \max_{x \geq 0} \mathbb{E}_{\xi \sim \mathrm{Exp}[\lambda]} [f(x,\xi)] = r \mathbb{E}_{\xi \sim \mathrm{Exp}[\lambda]} [\min(x,\xi)] - cx. \label{formu:newsvendor}
\end{align}

In the experiments, the true parameter of the exponential distribution is set to $\lambda^* = \frac{1}{120}$. An initial dataset of size $N_0 = 20$ is generated from the true distribution. During the sequential data collection process,
one additional observation is obtained at each iteration, i.e., $n_k = 1$. Each experiment is replicated 100 times to evaluate the performance of the stopping policies. 
\new{To ensure fair comparisons across policies, we use common random numbers so that all policies are evaluated under the same simulated data realizations within each replication.} 

\subsection{Analytical formulations} \label{subsec:analytical formulation}
\subsubsection{Optimal order of exponential Newsvendor}
At iteration $k$, we solve the Newsvendor problem using stochastic optimization to obtain the optimal order quantity $x_k$. In practice, the demand distribution is often not fully known. Instead, the underlying distribution $\mathbb{P}$ can be partially
characterized using a finite set of independent demand observations, which correspond to past realizations of the random demand. Let the set of empirical demand observations be denoted by $\mathcal{N} = \{\hat{\xi}_1, \hat{\xi}_2, \ldots, \hat{\xi}_N\}$, where $N$ represents the number of observed samples. 

\begin{lemma}[Expected minimum under exponential demand] \label{lem:min_demand}
If $\xi\sim \mathrm{Exp}(\lambda)$ and $x\ge 0$, then: 
\begin{align*}
    \mathbb{E}[\min(x,\xi)] = \frac{1-\exp^{- \lambda x}}{\lambda}.
\end{align*}
\end{lemma}

\begin{proof}   
    $\mathbb{E}_{\xi \sim \mathbb{P}} [\min(x,\xi)] = \int_0^\infty \min(x,\xi) f(\xi) d\xi$, where $f(x) = \lambda \exp^{- \lambda x} $ is the pdf of the exponential distribution. We have \[\mathbb{E}[\min(x, \xi)] = \int_{0}^{\infty} \min(x, \xi) \lambda e^{-\lambda \xi} d\xi = \int_{0}^{x} \xi \lambda e^{-\lambda \xi} d\xi + \int_{x}^{\infty} x \lambda e^{-\lambda \xi} d\xi. \] After taking the integrals of the two parts, we have $\mathbb{E}_{\xi \sim \mathbb{P}} [\min(x,\xi)] = \frac{1}{\lambda} (1- \exp^{- \lambda x})$.
\end{proof}

Hence, the expected objective value under rate $\lambda$ as below 
\begin{align*}
    \Pi(x;\lambda) = \mathbb{E}_{\xi \sim \mathrm{Exp}[\lambda]}[f(x,\xi)] = \frac{r}{\lambda} (1- \exp^{- \lambda x}) - cx. 
\end{align*}

\begin{lemma}[Optimal order for exponential newsvendor] \label{lem:optimal_order}
For $r>c$, the maximizer of $\Pi(x;\lambda)$ is
\begin{align*}
    x^*(\lambda)=\frac{1}{\lambda} \ln\frac{r}{c}. 
\end{align*}
Moreover, at $x^*(\lambda)$ we have $e^{-\lambda x^*(\lambda)}= \frac{c}{r}$. 
\end{lemma}

\begin{proof}
    See Appendix \ref{proof:optimal order}. 
\end{proof}

The solution $x_k$ and true optimal solution $x^*$ can be found by replacing the $\lambda$ with the estimated distribution $\lambda_k$ and true distribution $\lambda^*$, respectively. 

\begin{lemma}[One-step update of the posterior-mean rate estimator]\label{lem:posterior_mean_update}

Let $\xi \mid \lambda \sim \mathrm{Exp}(\lambda)$ and let the prior be $\lambda \sim \Gamma(\alpha_0,\beta_0)$ under the shape and rate parameterization. 
Given $N_k$ observations $\xi_{1:N_k}$ with sum $s_k:=\sum_{i=1}^{N_k}\xi_i$, define the posterior-mean estimator

\[ \lambda_k \;:=\; \mathbb{E}[\lambda \mid \xi_{1:N_k}] \;=\;\frac{\alpha_0+N_k}{\beta_0+s_k}. \]



When $n_k=1$, let $d$ be one additional observation. Then the posterior-mean estimator after incorporating $d$ is
\begin{align*}
    \lambda_{k+1}(d) :=\mathbb{E}[\lambda \mid \xi_{1:N_k},d] =\frac{\alpha_0+N_k+1}{\beta_0+ s_k+d} 
    =\lambda_k\cdot \frac{\alpha_0+N_k+1}{\alpha_0+N_k+\lambda_k d}.
\end{align*}

In particular, in the (improper) flat-prior limit $\alpha_0=0, \beta_0=0$, the posterior mean coincides with the MLE, and
\begin{align*}
    \tilde{\lambda}_{k+1}(d) =\frac{\lambda_k (N_k+1)}{N_k+\lambda_k d},
\end{align*}
where $d\sim \mathrm{Exp}(\lambda_k)$. 
\end{lemma}

\begin{proof}
    See Appendix \ref{proof:estimate update}. 
\end{proof}

\subsubsection{Analytical regret} 
Here, we derive the analytical formulation of the regret by taking the newsvendor problem formulation and optimal order in Lemma \ref{lem:optimal_order} under an exponential distribution. With the (expected) regret defined at iteration $k$ in Equation (\ref{formu:regret}), we drive the analytical regret formulation (\ref{eq:analyticalRE}) for the Newsvendor problem in  Proposition \ref{prop:RE_closed_form}. 

\begin{proposition}[Closed-form regret for exponential newsvendor] \label{prop:RE_closed_form}
Under the current setting, the regret admits a closed form
\begin{align}
   \mathrm{RE}(\lambda^*, S_k) = \frac{r}{\lambda^*}\left(\frac{c}{r}\right)^ \frac{\lambda^*}{\lambda_k} -\frac{c}{\lambda^*} + c\left(\frac{\lambda^* - \lambda_k}{\lambda_k \lambda^*}\right)\ln\frac{r}{c} \label{eq:analyticalRE}
\end{align}
\end{proposition}

\begin{proof}
    See Appendix \ref{proof:analytical regret}. 
\end{proof}

\subsubsection{Analytical benefit}
We now introduce the analytical formulation of real benefit after taking one additional data point.

\begin{proposition}[Analytical real benefit $RB(\lambda^*)$] \label{prop:RB_closed_form}
At iteration $k$, the decision $x_k$ is computed using the current estimate $\lambda_k$ as $ x_k = \frac{\ln\frac{r}{c}}{\lambda_k}$. Suppose $N_k$ data points have been observed so far. After collecting one additional data point $d \sim \mathrm{Exp}(\lambda^*)$, the updated estimator with $\alpha_0 =0$ and $\beta = 0$ is $\lambda_{k+1}(d) = \frac{\lambda_k (N_k+1)}{N_k + \lambda_k d}$, and the updated decision is $x_{k+1}(d)$. Then the real benefit admits the following closed-form expression (\ref{eq:analyticalRB}). 

\begin{align}
     & \mathrm{RB}(\lambda^*, S_k) = \frac{r}{\lambda^*}\left(\left(\frac{c}{r}\right)^ \frac{\lambda^*}{\lambda_k} -\left(\frac{c}{r}\right)^ \frac{\lambda^* N_k}{\lambda_k (N_k+1)} \frac{N_k+1}{N_k+1 - \ln \frac{c}{r}}\right) + c\left(\frac{\lambda^* - \lambda_k}{\lambda^* \lambda_k (N_k+1)}\ln\frac{r}{c}\right). \label{eq:analyticalRB}
\end{align}
\end{proposition}

\begin{proof}
    See Appendix \ref{proof:analytical benefit}. 
\end{proof}

Since the true demand distribution $\mathbb{P}[\lambda^*]$ is unknown in practice, we estimate the distribution parameter $\lambda^*$ using a Bayesian approach. The resulting estimate $\lambda$ is then substituted into the real benefit formulation in Equation~(\ref{eq:analyticalRB}) in place of $\lambda^*$, resulting in the formulation of $\mathrm{RB}(\lambda)$ as below: 

\begin{align}
     \mathrm{RB}(\lambda, S_k) = \frac{r}{\lambda}\left(\left(\frac{c}{r}\right)^ \frac{\lambda}{\lambda_k} -\left(\frac{c}{r}\right)^ \frac{\lambda N_k}{\lambda_k (N_k+1)} \frac{N_k+1}{N_k+1 - \ln \frac{c}{r}}\right) + c\left(\frac{\lambda- \lambda_k}{\lambda \lambda_k (N_k+1)}\ln\frac{r}{c}\right). \label{eq:analyticalEB}
\end{align}

\subsubsection{Mean of benefit} 
We then derive the formulation of mean real benefit ($MRB_k$) when the distribution parameter $\lambda$ follows a posterior distribution $\Gamma(\alpha_k, \beta_k)$ at the $k$-th iteration. 

\begin{proposition}[Analytical mean of estimated benefit $MRB$] \label{prop:meanRB_closed_form}
With Lemma~\ref{lem:min_demand}, Lemma~\ref{lem:optimal_order}, Lemma~\ref{lem:posterior_mean_update}, and Equation ~\ref{eq:analyticalEB}, the expected estimated benefit can be formulated analytically as 

\begin{align}
MRB_k 
= r\Bigl[\frac{\beta_k}{\alpha_k-1}\Bigl(\tfrac{\beta_k}{\beta_k - A}\Bigr)^{\alpha_k-1}
- C\frac{\beta_k}{\alpha_k-1}\Bigl(\tfrac{\beta_k}{\beta_k - B}\Bigr)^{\alpha_k-1}\Bigr]
+ \frac{c\ln\frac{r}{c}}{\lambda_kN_{k+1}}\Bigl(1-\lambda_k\frac{\beta_k}{\alpha_k-1}\Bigr), \label{eq:analytical MEB}
\end{align}

where $A=\frac{1}{\lambda_k}\ln\frac{c}{r},
\quad B=\frac{N_k}{\lambda_kN_{k+1}}\ln\frac{c}{r},
\quad C=\frac{N_{k+1}}{N_{k+1}-\ln\frac{c}{r}}$.
\end{proposition}

\begin{proof}
    See Appendix \ref{proof:analytical mean estimated benefit}. 
\end{proof}

We observe that the estimated real benefit may fluctuate across iterations and is not necessarily monotonically decreasing. Therefore, stopping immediately after a single low benefit estimate may lead to premature termination. To reduce this sensitivity to local fluctuations, we base the stopping decision on the average of the current estimated benefit and the benefit estimate from the previous iteration. We further apply this smoothing strategy to the proposed benefit-related stopping policies. 

In addition to the policies proposed in Section \ref{sec4:algorithm}, we include an oracle benchmark in the numerical study, referred to as Oracle. This benchmark evaluates the stopping decision using the true real benefit, computed under the true distributional parameter. Namely, the stopping indicator is $RB(\theta^*, S_k)$. In contrast, the proposed stopping policies rely on estimated distributional parameters obtained from the available data. Therefore, Oracle is not implementable in practice, but serves as a full-information reference that separates the effect of parameter estimation error from the stopping-rule design. We compare the implementable policies with this benchmark to quantify how much performance loss is caused by not knowing the true data-generating distribution.

\subsection{Stopping policy performance result} \label{subsec:stopping}
\new{
In this section, we empirically evaluate the performance of the five proposed stopping policies. We then compare them with three benchmarks, hindsight best, hindsight best fixed, and the Oracle stopping benchmark. The hindsight-best benchmark selects the best stopping time for each realized sample path, whereas the hindsight-best fixed benchmark selects the best common stopping time across all replications. The oracle benchmark assumes that the true distribution parameter is known and applies the stopping rule using the true real benefit.
}

The experimental settings are as follows: the maximum data collection budget is $B = 20$, the unit sampling cost is $c_s = 0.5$, and the credibility level is $95\%$ (i.e., $\gamma = 0.05$). The five policies are evaluated under identical random number streams in each replication. Each experiment is replicated 100 times. We report the mean and standard error of the performance metric $RESC_{\tau}$, where $\tau$ denotes the stopping iteration. The results are summarized in Table~\ref{tab:comparison results} and Figure~\ref{fig:stopping_results}.

\begin{table}[htb]
    \caption{Performance comparison results $RESC_\tau$ of stopping policies}
    \centering
    \begin{tabular}{ p{4cm} | p{4cm} | p{3cm} } 
         \hline
         \textbf{Stopping policy $\pi$} & \textbf{$RESC_{\tau_\pi}$}  & \textbf{Stopping time $\tau_\pi$} \\ \hline
          Fixed Budget (FB)             & 24.57 $\pm$ 0.70   &  40             \\ \hline 
          Parameter Mean (PM)           & 14.68  $\pm$ 2.60  &  0 \\  \hline 
          Parameter Interval (PI)       & 20.95 $\pm$ 0.81   &  31.08 $\pm$ 0.40 \\  \hline
          Measure Mean (MM)             & \textbf{12.23} $\pm$ \textbf{1.63} &  2.96 $\pm$ 0.21   \\  \hline
          Measure Interval (MI)         & 28.46 $\pm$ 0.64   & 48.92 $\pm$ 0.53   \\   \hline
          Oracle                        & 8.37 $\pm$ 0.99    & 4.01 $\pm$ 1.04  \\ \hline 
          Hindsight Best Bound          & 5.61 $\pm$ 0.82    & 5.35 $\pm$ 0.80 \\ \hline
    \end{tabular}
    \caption*{\footnotesize Note: results are reported as mean $\pm$ standard error. Bold values indicate the best performance among proposed stopping policies. The Oracle and Hindsight Best Bound are policy benchmarks.}
    \label{tab:comparison results}
\end{table}

\begin{figure}[htb]
    \centering
    \includegraphics[width=\textwidth]{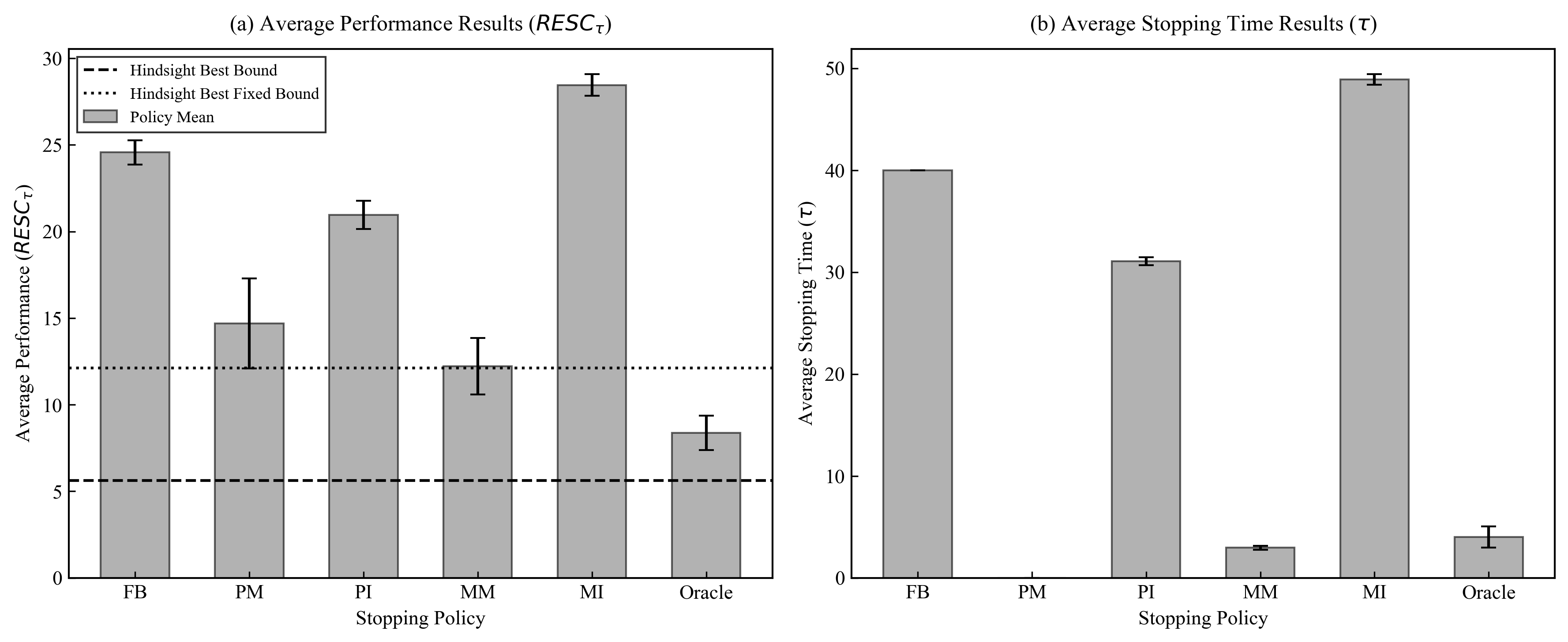}
    \caption{Performance Comparison of Stopping Policies. Panel (a) shows the performance metric $RESC_{\tau}$ relative to hindsight benchmarks (dashed lines). Panel (b) compares the average stopping time $\tau$ across policies. Error bars represent standard errors over 100 replications.}
    \label{fig:stopping_results}
\end{figure}

Several observations can be made from the results. First, the FB policy stops only when the data collection budget is exhausted, which corresponds to collecting \new{40 ($\frac{B}{c_s} = 40$) additional data points.} Second, among all policies, the MM policy achieves the best performance, yielding the smallest $RESC_{\tau}$. Third, the MI policy performs worst, since it stops by comparing the upper credibility bound of the estimated benefit with the unit sampling cost. Such a rule reflects a risk-averse decision maker who is willing to collect more data to reduce input uncertainty. Furthermore, we can see that the average stopping point of the MM policy (2.96 on average) is significantly smaller than that of the PI (31.08 on average). Besides, the average stopping result of the MM policy (12.23 on average) is less than that of the PI policy (20.69 on average). 

\new{To further evaluate the policies, we compare their performance with three benchmark bounds. 
The resulting value of the hindsight best bound is $5.35 \pm 0.80$. This bound represents the ideal performance that cannot be achieved in practice since it requires perfect hindsight. We can see that the MM policy achieves a performance closest to the hindsight best bound. We also compute the hindsight-best fixed benchmark, defined as the minimum average $RESC_{\tau}$ over all fixed stopping iterations. This benchmark achieves an average $RESC_{\tau}$ of 12.11. Our proposed MM and PM policies attain comparable performance, with values of 12.23 and 14.68, respectively. Notably, MM nearly matches the best fixed stopping rule chosen with hindsight, highlighting the value of adaptive stopping: even without hindsight information, the adaptive MM policy matches the performance of the best fixed stopping rule selected in hindsight. }

\new{We also compare the performance of proposed stopping policies with the Oracle benchmark. The Oracle policy takes the true distribution parameter $\theta^*$ when evaluating the real benefit of additional data collection. As expected, Oracle achieves an average ($RESC_{\tau}$) of (8.37) with an average stopping time of (4.01), outperforming all the proposed heuristic policies. Compared with MM, the best stopping policy, Oracle attains a lower ($RESC_{\tau}$) while stopping at a similar stage. This gap indicates that part of the performance loss of the proposed policies comes from uncertainty in estimating the true distribution parameter. However, Oracle still performs worse than the Hindsight Best Bound, which achieves $\mathrm{RESC}_\tau=5.61$. This is reasonable because Oracle has access to the true parameter but still makes stopping decisions sequentially, whereas the hindsight benchmark represents an ex-post lower bound. Therefore, Oracle serves as an intermediate benchmark between the proposed heuristic policies and the idealized hindsight best bound. Overall, these results indicate that the MM policy provides the most effective stopping decisions among the proposed policies.}

\subsection{Effect of Unknown True Parameter} \label{subsec:unknown true test}
In this subsection, we investigate the sensitivity of the stopping policies to the true distribution parameter $\lambda^*$. Specifically, we consider $\lambda^* \in \left\{\tfrac{1}{200},\tfrac{1}{160},\tfrac{1}{120}, \tfrac{1}{100},\tfrac{1}{80},\tfrac{1}{50}\right\}$. For each value of $\lambda^*$, the experiment is replicated 100 times. We report the average and standard error of the performance metric $RESC_{\tau_{\pi}}$ and the corresponding stopping time $\tau_{\pi}$. The results are presented in Figure~\ref{fig:results_vs_para} and Figure~\ref{fig:stop_vs_para}. 

\begin{figure}[htb!]
    \centering
    \includegraphics[width=1.0\textwidth]{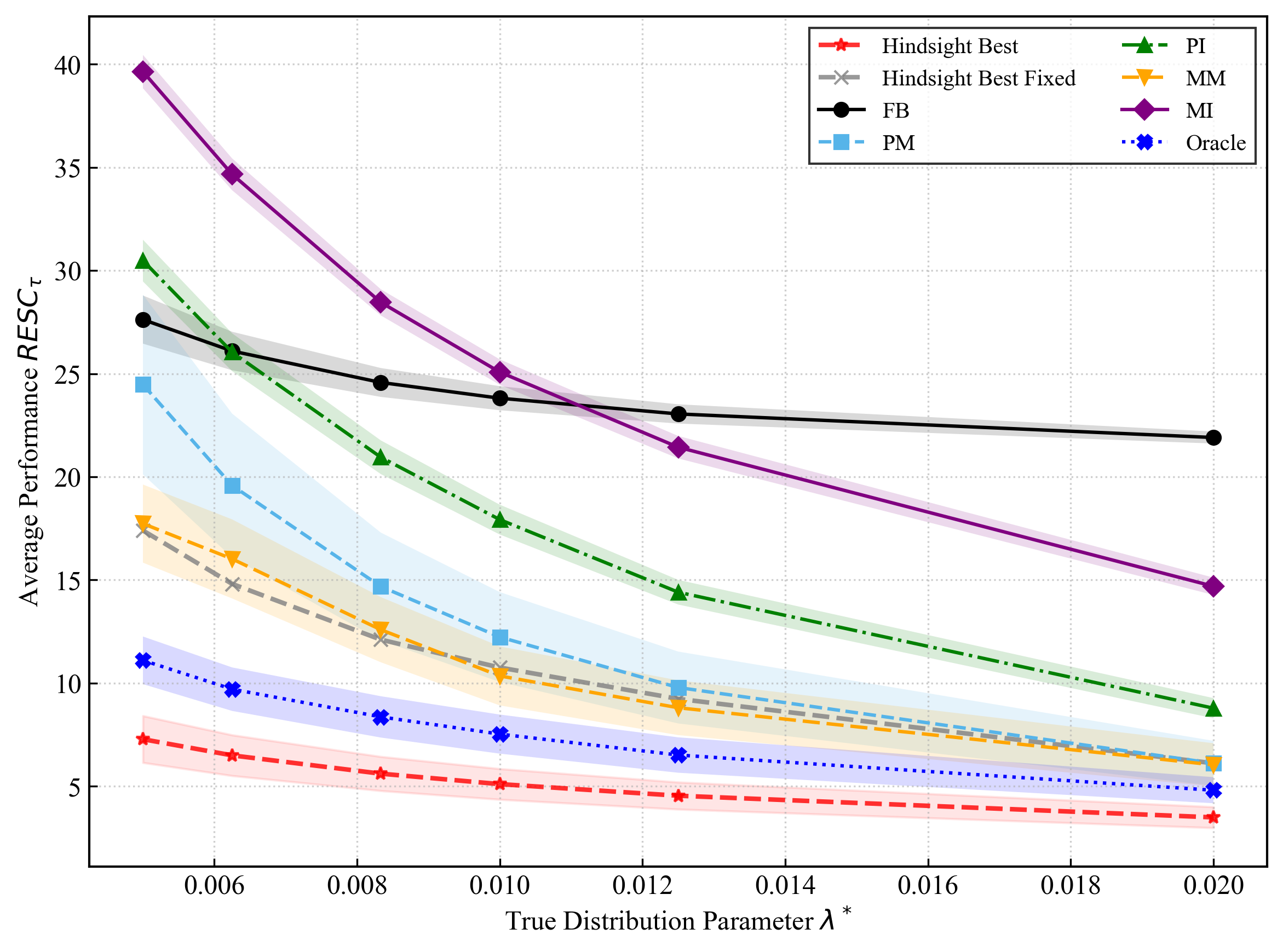}
    \caption{Average performance comparison of stopping policies $RESC_{\tau}$ under different values of true distribution parameter $\lambda^*$. The figure compares the average $RESC_\tau$ and standard errors (shaded area) for the five stopping policies (FB, PM, PI, MM, MI) against the Hindsight Best bound and Hindsight Best Fixed bound as the true parameter varies from a pre-identified set, $\lambda^* \in \left\{\tfrac{1}{200},\tfrac{1}{160},\tfrac{1}{120}, \tfrac{1}{100},\tfrac{1}{80},\tfrac{1}{50}\right\}$. Lower values indicate better performance.}
    \label{fig:results_vs_para}
\end{figure}

From Figure~\ref{fig:results_vs_para}, we can observe the following trends. First, all policies, including the hindsight-best benchmark, decrease monotonically as $\lambda^*$ increases. Since $\lambda^*$ is the reciprocal of the mean demand in the exponential distribution, a larger $\lambda^*$ corresponds to a smaller mean demand level. This reduces the magnitude of decision risk and leads to a smaller regret–sampling cost trade-off. However, 
how quickly regret approaches zero differs substantially across policies. The fixed-budget (FB) policy shows only a mild decrease and remains consistently dominated, reflecting its lack of adaptive stopping capability. In contrast, the adaptive policies (PI, PM, MM, and MI) demonstrate significantly stronger sensitivity to the problem parameter. Among the adaptive policies, the MM and PM policies consistently achieve the strongest performance as $\lambda^*$ increases and remain closest to the hindsight-best benchmark across most parameter regimes. In particular, the MM policy consistently outperforms the other feasible policies throughout the tested range of parameters. It achieves the lowest $RESC_{\tau}$ among all policies and maintains the smallest gap to the Hindsight Best bound. The MI policy exhibits relatively high $RESC_{\tau}$ when $\lambda^*$ is small (i.e., when the mean demand is large), suggesting potential over-exploration in highly uncertain environments. Nevertheless, its performance improves substantially as $\lambda^*$ increases. 

\new{The Oracle benchmark consistently outperforms all adaptive policies across the tested values of $\lambda^*$, showing the performance gain that could be achieved if the true distribution parameter were known. However, Oracle remains above the Hindsight Best benchmark because it still follows a sequential stopping rule rather than relying on ex-post information. The gap between Oracle and Hindsight Best narrows as $\lambda^*$ increases, indicating that the distribution parameter uncertainty has a larger effect when the mean demand is higher. Overall, these results demonstrate that adaptive, information-driven stopping policies are more sensitive to variations in the underlying distribution parameter, whereas non-adaptive policies suffer from persistent inefficiency across parameter regimes.} 

\begin{figure}[htb!]
    \centering
    \includegraphics[width=1.0\textwidth]{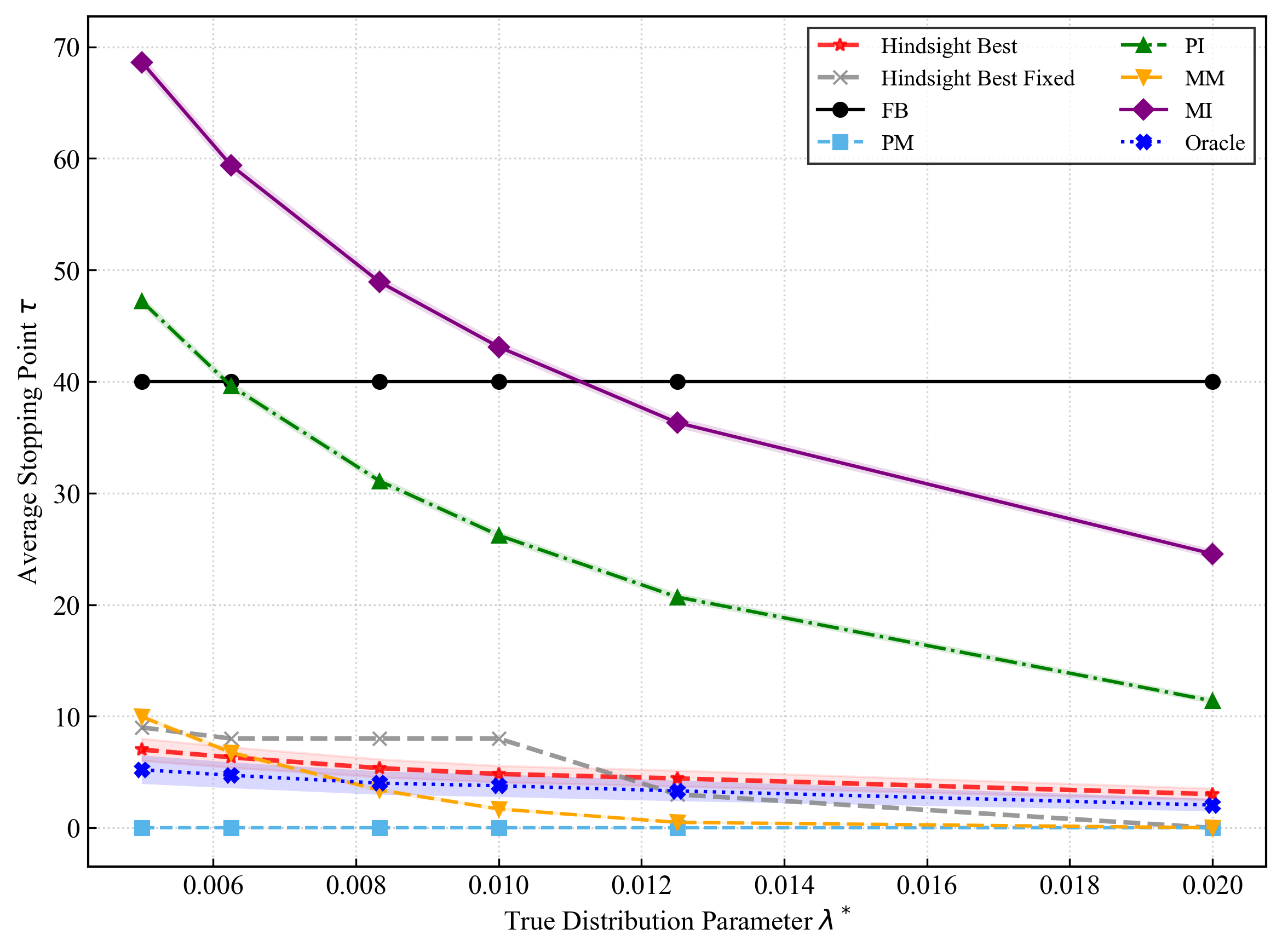}
    \caption{Average stopping time $\tau$ of each policy under different true distribution parameters $\lambda^*$. The figure compares the average stopping point $\tau$ and standard errors (shaded area) for the five stopping policies (FB, PM, PI, MM, MI) against the Hindsight Best bound and Hindsight Best Fixed bound as the true parameter varies from a pre-identified set, $\lambda^* \in \left\{\tfrac{1}{200},\tfrac{1}{160},\tfrac{1}{120}, \tfrac{1}{100},\tfrac{1}{80},\tfrac{1}{50}\right\}$. Lower values indicate stopping earlier.}
    \label{fig:stop_vs_para}
\end{figure}

From Figure~\ref{fig:stop_vs_para}, we can see several patterns. In terms of the stopping time, Oracle stops early across all values of $\lambda^*$ and its stopping point decreases as $\lambda^*$ increases. As expected, the stopping time of the FB policy remains constant across all parameter values. 

\new{For large $\lambda^*$, regret is less sensitive to parameter estimation error. Consequently, the adaptive policies require fewer observations before stopping data collection.} Third, the PI and MI policies consistently stop later than the other adaptive policies. This behavior arises from their credibility-interval–based stopping rules, which incorporate uncertainty in the parameter or benefit estimate. As a result, these policies tend to collect more data to ensure a reliable decision, reflecting a more conservative exploration strategy. While this conservative strategy reduces the risk of premature stopping, it also leads to higher sampling costs and therefore larger $RESC_\tau$ values, as shown in Figure~\ref{fig:results_vs_para}. 

Finally, the average stopping time of the MM policy decreases rapidly as $\lambda^*$ increases and approaches zero when the demand uncertainty becomes sufficiently small. \new{In addition, the MM policy terminates data collection significantly earlier than the other adaptive policies, except PM, across the entire range of $\lambda^*$.} Despite this aggressive stopping behavior, the MM policy consistently achieves the best trade-off between sampling cost and decision quality among the proposed policies, as the corresponding $RESC_\tau$ values remain very close to the hindsight-best benchmark in Figure~\ref{fig:results_vs_para}. This indicates that the MM policy is able to effectively identify when the marginal value of additional data becomes negligible, thereby avoiding unnecessary sampling while maintaining near-optimal decision performance. 

Taken together, the results on the performance metric $RESC_\tau$ and the stopping time $\tau$ reveal a clear relationship between sampling efficiency and policy performance. As the true parameter $\lambda^*$ increases, the average stopping time of the adaptive policies decreases substantially, indicating that fewer observations are required to reach a reliable decision when the underlying demand uncertainty becomes smaller. This reduction in sampling effort directly translates into improved performance, as reflected by the decreasing values of $RESC_\tau$ across all adaptive policies. Policies that directly evaluate the expected marginal value of additional data, such as the MM policy, are able to adapt their sampling effort more efficiently and achieve a near-optimal balance between decision quality and data collection cost.

\subsection{Effect of Data Collection Cost} \label{subsec:sampling cost test}
In the baseline setting, the unit sampling cost is set to $c_s = 0.5$. To examine the sensitivity of the stopping policies to sampling cost, we conduct experiments over a range of values $c_s \in [5\times10^{-2}, 10^{-1}, 0.5, 1.0, 2.0, 4.0, 5.0, 10.0, 20.0]$. The corresponding results are presented in Figure~\ref{fig:results_vs_samplecost}.

\begin{figure}[htb!]
    \centering
    \includegraphics[width=1.0\textwidth]{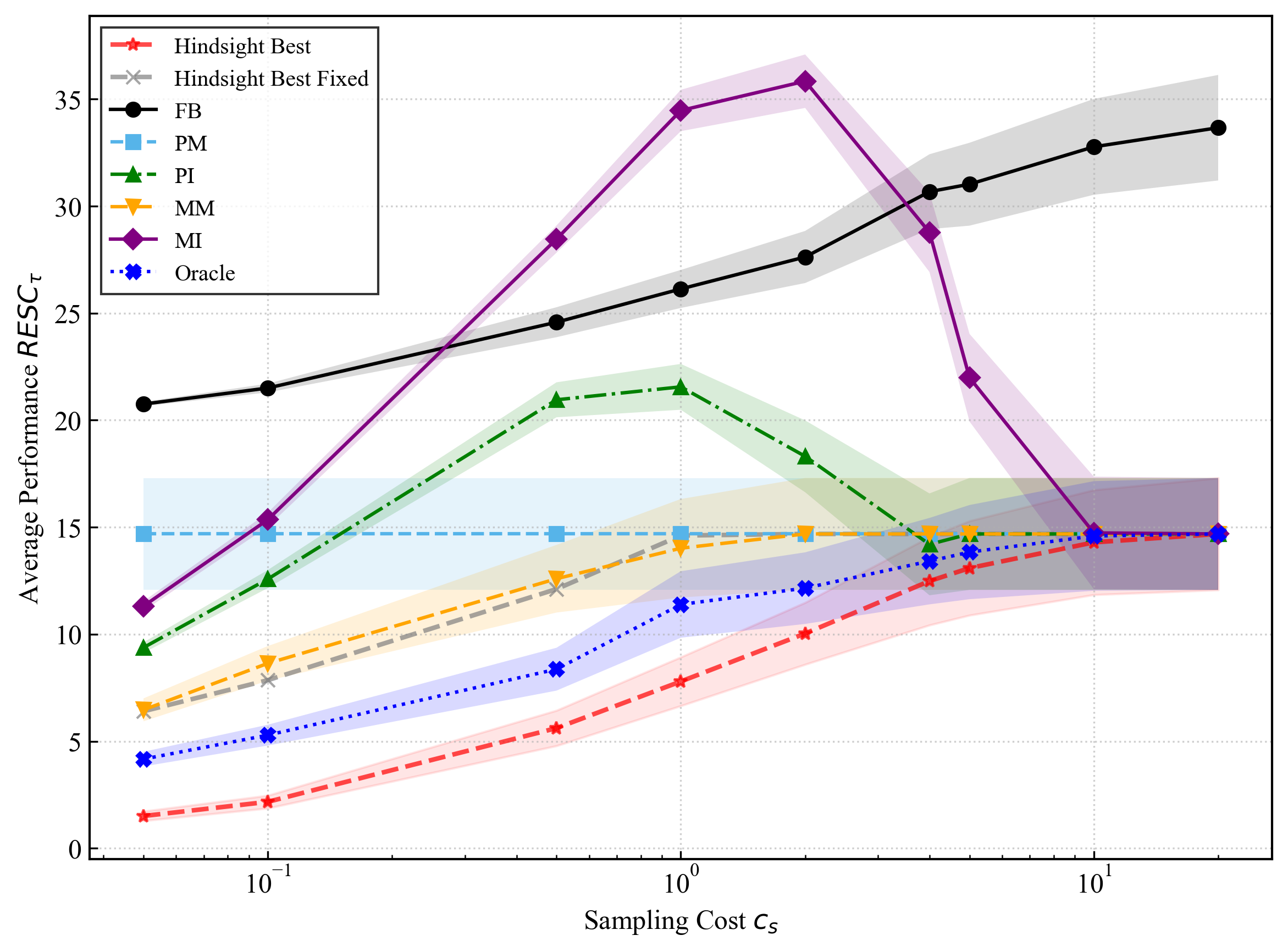}
    \caption{Performance ($RESC_\tau$) Comparison of Stopping Policies under different unit sampling cost $c_s$. The figure compares the average $RESC_\tau$ and standard errors for the five stopping policies (FB, PM, PI, MM, MI) against the Hindsight Best bound as the unit sampling cost varies from a pre-identified set, $[5\times10^{-2}, 10^{-1}, 0.5, 1.0, 2.0, 4.0, 5.0, 10.0, 20.0]$. Lower values indicate better performance.}
    \label{fig:results_vs_samplecost}
\end{figure}

\new{As expected, the $RESC_{\tau}$ of the hindsight-best benchmark increases monotonically with $c_s$, reflecting the direct contribution of sampling cost to the total regret by making data collection more expensive. The hindsight-best fixed benchmark remains relatively close to the proposed MM policy, while the Oracle consistently outperforms all adaptive stopping policies across the range of $c_s$. The performance of Oracle lies between the hindsight-best benchmark and the MM policy. } 

\new{The non-adaptive FB policy exhibits an increasing trend and remains consistently suboptimal under different cost regimes. As the sampling cost becomes sufficiently high, the performance gap between FB and the adaptive policies widens as $c_s$ increases, highlighting the inefficiency of non-adaptive data collection strategies when sampling becomes expensive.} The performance of the PM policy remains unchanged across different sampling costs. This occurs because the stopping rule under the PM policy never triggers additional data collection. In particular, the estimated marginal benefit under this policy remains below the unit sampling cost throughout the experiment, leading the policy to stop immediately after the initial data are observed. In contrast, the other adaptive policies (PI, MM, and MI) exhibit clear sensitivity to the sampling cost. When $c_s$ is small, the cost of acquiring additional data is relatively negligible, and policies that favor more aggressive exploration (such as MI and PI) can improve decision quality through additional sampling, resulting in lower values of $RESC_\tau$. \new{As $c_s$ increases, the cost of data collection increase; however, the marginal value of additional information remains the same. Under such conditions, policies that more carefully balance information gain and sampling cost (particularly the MM policy) demonstrate more stable and robust performance. The performance of MI and PI improves substantially as $c_s$ is sufficiently high. This pattern suggests that PI and MI tends to collect too much data when sampling is inexpensive, while its more conservative behavior becomes advantageous when data acquisition is costly.} 
We also observe that when the sampling cost is sufficiently high, all policies stop collecting additional data immediately, and their performance therefore converges to the same value as the three benchmarks. This is because the marginal benefit of collecting additional data is no longer sufficient to offset the high sampling cost under any policy.

\begin{figure}[htb!]
    \centering
    \includegraphics[width=1.0\textwidth]{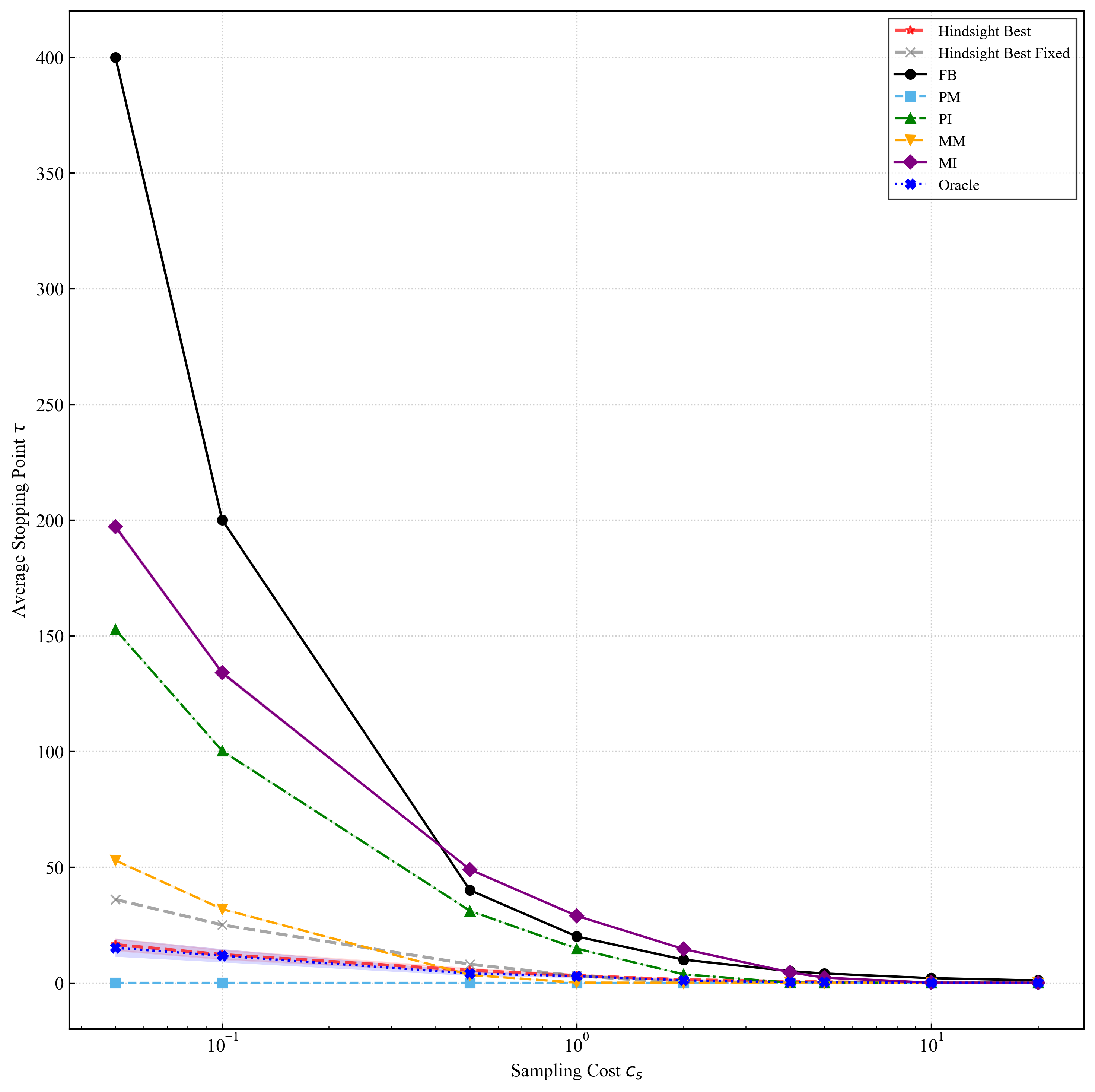}
    \caption{Stopping Point ($\tau$) Comparison of Stopping Policies under different unit sampling cost $c_s$. The figure compares the average $\tau$ and standard errors for the five stopping policies (FB, PM, PI, MM, MI) against the Hindsight Best and Hindsight Best Fixed bound as the unit sampling cost varies from a pre-identified set, $[5\times10^{-2}, 10^{-1}, 0.5, 1.0, 2.0, 4.0, 5.0, 10.0, 20.0]$. Lower values indicate early stopping.}
    \label{fig:stop_vs_samplecost}
\end{figure}

We also report the average stopping time $\tau$ under different unit sampling costs $c_s$ in Figure~\ref{fig:stop_vs_samplecost}. 
\new{As $c_s$ increases, the stopping time of all three benchmarks decrease. The Oracle stops at nearly the same time as the hindsight best benchmark, whereas the hindsight best fixed benchmark stops later. The relatively early stopping of the Oracle suggests that knowing the true distribution helps the policy evaluate the value of additional data more accurately than the adaptive stopping policies, thereby avoiding unnecessary sampling.}

Several observations can be made. First, the stopping time of all adaptive policies decreases sharply as the sampling cost increases. When the sampling cost is small, the marginal cost of acquiring additional data is negligible, which incentivizes policies to collect more observations before terminating the data collection process. As a result, the stopping points are relatively large in the low-cost regime. As $c_s$ increases, however, the cost of additional sampling becomes increasingly dominant, and all policies respond by terminating the data collection process earlier. Second, the FB policy exhibits the largest stopping times when the sampling cost is small. Since the FB policy continues sampling until the sampling budget is exhausted, it collects a large number of observations when the cost of sampling is low. As the sampling cost increases, however, the budget constraint becomes more restrictive, which substantially reduces the number of additional samples that can be collected.

Third, the adaptive policies display different exploration behaviors. The MI and PI policies tend to collect substantially more samples than the other adaptive policies when sampling costs are low, reflecting their more conservative stopping criteria that emphasize reducing parameter uncertainty. In contrast, the MM policy consistently stops earlier than PI and MI across most cost regimes, indicating that it more aggressively terminates sampling once the marginal benefit of additional information becomes small. Finally, when the sampling cost becomes sufficiently large (for example, $c_s = 20.0$), all policies converge to very small stopping times (nearly stopping immediately). In this regime, the cost of additional data dominates any potential benefit from improved parameter estimation, making early termination the economically rational decision. 

Overall, the results illustrate the fundamental tradeoff between learning and sampling cost. As the unit sampling cost increases, the optimal strategy shifts from exploration-intensive data collection toward earlier stopping. Policies that continue sampling excessively in low-cost regimes incur unnecessary sampling costs, while policies that terminate too early risk insufficient learning. The MM policy achieves the best compromise by avoiding excessive sampling in low-cost regimes while still terminating early when sampling becomes expensive. Across most cost regimes, the adaptive stopping policies achieves performance comparable to hindsight best fixed benchmark by dynamically adjusting their sampling decisions to the economic value of additional information. To conclude, the MM policy is the best stopping policy among the other policies since it consistently attains strong $RESC_\tau$ performance across a wide range of sampling cost levels.


\section{Conclusion and discussion} \label{sec7:conclusion}
This paper studies a sequential optimal stopping problem for data collection in stochastic programming under input uncertainty. We developed a regret-based framework that explicitly balances the expected value of additional information against the cost of acquiring more data. By modeling the evolution of uncertainty in a Bayesian learning framework and embedding it into a dynamic programming formulation, we characterize the stopping decision through a Bellman optimality condition. This condition admits a natural economic interpretation: stopping incurs the current expected regret, whereas continuing incurs an additional sampling cost together with the expected future regret. \new{Building on this principle, we propose five stopping policies,the nonadaptive Fixed Budget (FB) policy and four adaptive policies: the Parameter Mean (PM), Parameter Interval (PI), Measure Mean (MM), and Measure Interval (MI) policies. We also propose three benchmarks for policy performance comparison: hindsight best, hindsight best fixed, and the oracle stopping benchmark.}


The numerical experiments further demonstrate the practical value of the proposed approach. Across a broad range of settings, the adaptive stopping policies consistently outperform the non-adaptive fixed-budget rule by avoiding unnecessary data collection while maintaining \new{near hindsight best fixed benchmark decision performance. Among the proposed methods, the MM and PM policies deliver the strongest overall performance and remain closest to the three benchmarks across different values of the true distribution parameter and the unit sampling cost. Particularly, the proposed MM policy nearly matches the hindsight best fixed benchmark. This result demonstrates the value of adaptive data collection: our framework can dynamically determine when to stop sampling and achieve performance comparable to a hindsight-informed best fixed policy, without requiring the data collection budget to be specified in advance.} Overall, the results show that adaptive, information-driven stopping rules are more responsive to changes in the underlying environment and offer an effective approach for sequential data collection in stochastic optimization.

Several directions remain for future research. First, the current framework relies on a one-step look-ahead approximation, which evaluates the value of collecting only one additional observation at each stage. \new{Although this design yields tractable and effective policies, it may overlook situations in which the value of information emerges only after multiple additional samples are collected, which could arise in other problems.} Extending the framework to two-step or multi-step look-ahead policies is therefore a natural next step. Second, we assume a constant unit sampling cost, whereas many practical settings involve heterogeneous, state-dependent, or time-varying costs. Incorporating such cost structures would broaden the applicability of the model. Finally, applying the proposed framework to a wider range of domains, such as clinical trial design, and other data-intensive decision problems, would provide further empirical validation and help reveal additional modeling and computational challenges.

\newpage
\appendix
\section{Proof for Section \ref{sec6:numericalstudy}} 

\subsection{Proof of Lemma~\ref{lem:optimal_order}}\label{proof:optimal order}
\begin{proof} [Proof of Lemma~\ref{lem:optimal_order}]
With the objective formulation as 
\begin{align*}
    \mathbb{E}_{\xi \sim \mathrm{Exp}[\lambda]}[f(x,\xi)] = \frac{r}{\lambda} (1- \exp^{- \lambda x}) - cx. 
\end{align*}

Differentiate the objective formulation, then we get: 
\[
\frac{\partial}{\partial x}\Pi(x;\lambda)=re^{-\lambda x}-c.
\]
Setting to zero gives $re^{-\lambda x}=c$, i.e., $x=L/\lambda$, where $L=\ln(r/c)$. The second derivative is $-\lambda r e^{-\lambda x}<0$, so this is a maximum.
\end{proof}

\subsection{Proof of Lemma~\ref{lem:posterior_mean_update}}\label{proof:estimate update}
\begin{proof} [Proof of Lemma~\ref{lem:posterior_mean_update}] 
Under $\xi_i\mid \lambda {\sim}\mathrm{Exp}(\lambda)$, the likelihood for $\xi_{1:N_k}$ is
\[ p(\xi_{1:N_k}\mid \lambda)\propto \lambda^{N_k}\exp\!\big(-\lambda s_k\big). \]
Combining with the Gamma prior density $p(\lambda)\propto \lambda^{\alpha_0-1}e^{-\beta_0\lambda}$ yields the posterior
\[ p(\lambda\mid \xi_{1:N_k}) \propto \lambda^{\alpha_0+N_k-1}\exp\!\big(-(\beta_0+s_k)\lambda\big), \]
i.e., $\lambda\mid \xi_{1:N_k}\sim \Gamma(\alpha_0+N_k,\beta_0+s_k)$.

Hence, the posterior mean is
\[ \lambda_k=\mathbb{E}[\lambda\mid \xi_{1:N_k}] = \frac{\alpha_k}{\beta_k} = \frac{\alpha_0+N_k}{\beta_0+s_k}. \]

$d$ is generated as $d\sim \mathrm{Exp}(\lambda_k)$. After observing one additional sample $d$, the updated sufficient statistics become
$N_{k+1}=N_k+1$ and $S_{k+1}=s_k+d$, so by conjugacy
\[ \lambda\mid \xi_{1:N_k},d \sim \Gamma(\alpha_0+N_k+1,\beta_0+s_k+d), \]
and therefore \[ \lambda_{k+1}(d) =\frac{\alpha_0+N_k+1}{\beta_0+s_k+d}. \]

To express $\lambda_{k+1}(d)$ in terms of $\lambda_k$, note that $\lambda_k=(\alpha_0+N_k)/(\beta_0+s_k)$, we have: 

\begin{align*}
    \lambda_{k+1}(d) :=\mathbb{E}[\lambda \mid \xi_{1:N_k},d] =\frac{\alpha_0+N_k+1}{\beta_0+s_k+d} 
    =\lambda_k\cdot \frac{\alpha_0+N_k+1}{\alpha_0+N_k+\lambda_k d}.
\end{align*}
\end{proof}

\subsection{Proof of Proposition~\ref{prop:RE_closed_form}}\label{proof:analytical regret}

\begin{proof} [Proof of Proposition~\ref{prop:RE_closed_form}] 
For any $\lambda>0$, the expected objective value under $\xi\sim \mathrm{Exp}(\lambda)$ is 
\[\Pi(x;\lambda) :=\mathbb{E}_{\xi\sim \mathrm{Exp}(\lambda)}[f(x,\xi)] = r\frac{1-e^{-\lambda x}}{\lambda}-cx. \]

Therefore, \[ \mathrm{RE}(\lambda^*, S_k) = \Pi(x^*;\lambda^*)-\Pi(x_k;\lambda^*). \]

We first evaluate $\Pi(x^*;\lambda^*)$. Since $x^* = L/\lambda^*$, we have $e^{-\lambda^* x^*}=e^{-L}=c/r$. Hence
\begin{align*}
\Pi(x^*;\lambda^*)
&= r\frac{1-e^{-\lambda^* x^*}}{\lambda^*}-cx^*\\
&= r\frac{1-\frac{c}{r}}{\lambda^*}-c\frac{L}{\lambda^*}
= \frac{r-c}{\lambda^*}-\frac{cL}{\lambda^*}.
\end{align*}

We then evaluate the second term $\Pi(x_k;\lambda^*)$
Using $x_k=L/\lambda_k$,
\[\Pi(x_k;\lambda^*) = r\frac{1-e^{-\lambda^* x_k}}{\lambda^*}-cx_k. \]

Moreover,
\[ e^{-\lambda^* x_k} = e^{-\lambda^* L/\lambda_k} = \big(e^{-L}\big)^{\lambda^*/\lambda_k} = \left(\frac{c}{r}\right)^{\lambda^*/\lambda_k}. \]

Thus, \[\Pi(x_k;\lambda^*) = \frac{r}{\lambda^*}\left(1-\left(\frac{c}{r}\right)^{\lambda^*/\lambda_k}\right) -c\frac{L}{\lambda_k}. \]

We then subtract and simplify the analytical formulation for the regret. 
Compute
\begin{align*}
\mathrm{RE}(\lambda^*, S_k) 
& =\Pi(x^*;\lambda^*)-\Pi(x_k;\lambda^*) \\
& = \left(\frac{r-c}{\lambda^*}-\frac{cL}{\lambda^*}\right) - \left[ \frac{r}{\lambda^*}\left(1-\left(\frac{c}{r}\right)^{\lambda^*/\lambda_k}\right) -c\frac{L}{\lambda_k} \right]\\
& = \frac{r}{\lambda^*}\left(\frac{c}{r}\right)^{\lambda^*/\lambda_k} -\frac{c}{\lambda^*} +\left(c\frac{L}{\lambda_k}-c\frac{L}{\lambda^*}\right).
\end{align*}

Finally,
\[ c\frac{L}{\lambda_k}-c\frac{L}{\lambda^*}
= cL\left(\frac{1}{\lambda_k}-\frac{1}{\lambda^*}\right)
= c\left(\frac{\lambda^*-\lambda_k}{\lambda_k\lambda^*}\right)\ln\!\Big(\frac{r}{c}\Big), \]

which yields the closed form for the regret. 
\end{proof}

\subsection{Proof of Proposition~\ref{prop:RB_closed_form}}\label{proof:analytical benefit}

\begin{proof} [Proof of Proposition~\ref{prop:RB_closed_form}]
Starting from the real benefit definition, the inner expectation over $\xi$ can be written using the expected profit function
\[
\Pi(x;\lambda) := \mathbb{E}_{\xi\sim \mathrm{Exp}(\lambda)}[f(x,\xi)]
= r\frac{1-e^{-\lambda x}}{\lambda}-cx.
\]
Hence,
\begin{align*}
   \mathrm{RB}(\lambda^*, S_k)
    & = \mathbb{E}_{d\sim \mathrm{Exp}(\lambda^*)} \Big[\Pi(x_{k+1}(d);\lambda^*) - \Pi(x_k;\lambda^*) \Big]. 
\end{align*}
Recall that $L:=\ln\frac{r}{c}>0$, and the current decision is \[x_k=\frac{L}{\lambda_k}.\]

After observing one additional data point $d$, the updated estimator is
\[ \lambda_{k+1}(d)=\frac{\lambda_k (N_k+1)}{N_k+\lambda_k d}, \qquad d\sim \mathrm{Exp}(\lambda^*), \]
and therefore the updated decision satisfies
\begin{align*}
    \frac{1}{\lambda_{k+1}(d)} = \frac{1}{\lambda_k}\cdot \frac{N_k+\lambda_k d}{N_k+1}.
\end{align*}

We first evaluate $\Pi(x_{k+1}(d);\lambda^*)$. Using the above expression for $x_{k+1}(d)$, we have
\begin{align*}
   \Pi(x_{k+1}(d);\lambda^*) = \frac{r}{\lambda^*}(1-e^{-\lambda^*x_{k+1}(d)}) - cx_{k+1}(d)
\end{align*}

We then evaluate $\Pi(x_k;\lambda^*)$: 
\begin{align*}
    \Pi(x_k;\lambda^*) = \frac{r}{\lambda^*}(1-e^{-\lambda^* x_{k}}) - cx_{k}
\end{align*}

Subtracting the two terms yields 
\begin{align*}
    \Pi(x_{k+1}(d);\lambda^*)-\Pi(x_k;\lambda^*) = \frac{r}{\lambda^*}(e^{-\lambda^* x_{k}} -e^{-\lambda^* x_{k+1}(d)}) + c\left(x_{k} - x_{k+1}(d)\right)
\end{align*}

With $x_k = \frac{\ln\frac{r}{c}}{\lambda_k}$, $x_{k+1}(d) = \frac{\ln\frac{r}{c}}{\lambda_{k+1}(d)}$, and the relationship between $\lambda_{k+1}(d)$ and $\lambda_k$, we have: 
\begin{align*}
    \Pi(x_{k+1}(d);\lambda^*)-\Pi(x_k;\lambda^*) = \frac{r}{\lambda^*}((\frac{c}{r})^ \frac{\lambda^*}{\lambda_k} -(\frac{c}{r})^ \frac{\lambda^* (N_k + \lambda_k d)}{\lambda_k (N_k+1)}) + c(\frac{1-\lambda_k d}{\lambda_k (N_k +1) })\ln\frac{r}{c}
\end{align*}

Then we take the expectation of the difference since $ d\sim \mathrm{Exp}(\lambda^*)$, we have: 
\begin{align*}
    RB(\lambda^*, S_k) = & \mathbb{E}_{d \sim \mathbb{P}[\lambda^*]} [\Pi(x_{k+1}(d);\lambda^*)-\Pi(x_k;\lambda^*)] \\
    = & \mathbb{E}_{d \sim \mathbb{P}[\lambda^*]} [ \frac{r}{\lambda^*}((\frac{c}{r})^ \frac{\lambda^*}{\lambda_k} -(\frac{c}{r})^ \frac{\lambda^* (N_k + \lambda_k d)}{\lambda_k (N_k+1)}) + c(\frac{1-\lambda_k d}{\lambda_k (N_k+1)})\ln\frac{r}{c}] \\
    = & \frac{r}{\lambda^*} \left( \left(\frac{c}{r}\right)^{\frac{\lambda^*}{\lambda_k}} - \left(\frac{c}{r}\right)^{\frac{\lambda^* N_k}{\lambda_k (N_k+1)}} \frac{N_k+1}{N_k+1-\ln\!\left(\frac{c}{r}\right)} \right) + c\left(\frac{\lambda^* - \lambda_k}{\lambda^* \lambda_k(N_k+1)}\right)\ln\!\left(\frac{r}{c}\right). 
\end{align*}

\end{proof}

\subsection{Proof of Proposition~\ref{prop:meanRB_closed_form}}\label{proof:analytical mean estimated benefit}

\begin{proof}[Proof of Proposition~\ref{prop:meanRB_closed_form}] 
When the true distribution parameter is unknown, we have estimated benefit $RB(\lambda, S_k)$. We have
\[
\mathrm{RB}(\lambda, S_k) = \frac{r}{\lambda}((\frac{c}{r})^ \frac{\lambda}{\lambda_k} -(\frac{c}{r})^ \frac{\lambda N_k}{\lambda_k N_{k+1}} \frac{N_{k+1}}{N_{k+1} - \ln \frac{c}{r}}) + c\ln\frac{r}{c}(\frac{\lambda - \lambda_k}{\lambda \lambda_k N_{k+1}})
\]
Assume the distribution parameter $\lambda$ follows a posterior Gamma distribution
\[
\lambda \sim \mathrm{Gamma}(\alpha,\beta),
\quad f_\lambda(\lambda)=\frac{\beta^\alpha}{\Gamma(\alpha)}\,\lambda^{\alpha-1}e^{-\beta\lambda}.
\]

Define
\[
A=\frac{1}{\lambda_k}\ln\frac{c}{r},
\quad B=\frac{N_k}{\lambda_kN_{k+1}}\ln\frac{c}{r},
\quad C=\frac{N_{k+1}}{N_{k+1}-\ln(c/r)},
\quad D=c\ln\frac{r}{c}.
\]
Hence \[ (\frac{c}{r})^\frac{\lambda}{\lambda_k} = e^{A\lambda},
\quad
(\frac{c}{r})^ \frac{\lambda N_k}{\lambda_k N_{k+1}} \frac{N_{k+1}}{N_{k+1} - \ln \frac{c}{r}} = Ce^{B\lambda}
\]
Then
\[
\mathrm{RB}(\lambda, S_k)
=\frac{r}{\lambda}\bigl(e^{A\lambda}-C\,e^{B\lambda}\bigr)
+ D\frac{\lambda-\lambda_k}{\lambda\,\lambda_k\,N_{k+1}}.
\]

We first compute
\[
MRB = \mathbb{E}[\mathrm{RB}(\lambda, S_k)]
=\int_0^\infty \mathrm{RB}(\lambda_k,\lambda)\,f_\lambda(\lambda)\,d\lambda.
\]
Split into terms:
\[
E_1 = r\int_0^\infty \frac{1}{\lambda}e^{A\lambda}f_\lambda(\lambda)d\lambda,
\quad
E_2 = rC\int_0^\infty \frac{1}{\lambda}e^{B\lambda}f_\lambda(\lambda)d\lambda,
\]
\[
E_3 = D\int_0^\infty \frac{\lambda-\lambda_k}{\lambda\,\lambda_k\,N_{k+1}}f_\lambda(\lambda)d\lambda.
\]

Using the standard integral
\[
\int_0^\infty \lambda^{\nu-1}e^{-s\lambda}\,d\lambda=\frac{\Gamma(\nu)}{s^\nu},
\quad \Re(s)>0,\;\Re(\nu)>0,
\]
we have for $a<\beta$:
\[
\int_0^\infty \frac{1}{\lambda}e^{a\lambda}f_\lambda(\lambda)d\lambda
=\frac{\beta^\alpha}{\Gamma(\alpha)}\int_0^\infty \lambda^{\alpha-2}e^{-(\beta-a)\lambda}d\lambda
=\frac{\beta}{\alpha-1}\Bigl(\frac{\beta}{\beta-a}\Bigr)^{\alpha-1}.
\]
Thus
\[
E_1 = r\frac{\beta}{\alpha-1}\Bigl(\frac{\beta}{\beta - A}\Bigr)^{\alpha-1},
\quad
E_2 = rC\frac{\beta}{\alpha-1}\Bigl(\frac{\beta}{\beta - B}\Bigr)^{\alpha-1}.
\]
For $E_3$, note
\[
\int_0^\infty \frac{1}{\lambda}f_\lambda(\lambda)d\lambda=\frac{\beta}{\alpha-1},
\quad
\int_0^\infty f_\lambda(\lambda)d\lambda=1.
\]
Hence
\[
E_3 = \frac{D}{\lambda_k N_{k+1}}\Bigl(1 - \lambda_k\frac{\beta}{\alpha-1}\Bigr).
\]
Combining,
\[
MRB = \mathbb{E}_{\lambda \sim \Gamma(\alpha, \beta)}[\mathrm{RB}(\lambda, S_k)]
= r\Bigl[\frac{\beta}{\alpha-1}\Bigl(\tfrac{\beta}{\beta - A}\Bigr)^{\alpha-1}
- C\frac{\beta}{\alpha-1}\Bigl(\tfrac{\beta}{\beta - B}\Bigr)^{\alpha-1}\Bigr]
+ \frac{c\ln(r/c)}{\lambda_kN_{k+1}}\Bigl(1-\lambda_k\frac{\beta}{\alpha-1}\Bigr).
\]

\end{proof}

\newpage

\bibliographystyle{agsm}
\bibliography{references}

@article{xu2023decision,
  title={Decision making under costly sequential information acquisition: the paradigm of reversible and irreversible decisions},
  author={Xu, Renyuan and Zariphopoulou, Thaleia and Zhang, Luhao},
  journal={arXiv preprint arXiv:2401.00569},
  year={2023}
}

@article{fu2010endogenous,
  title={Endogenous information acquisition in supply chain management},
  author={Fu, Qi and Zhu, Kaijie},
  journal={European Journal of Operational Research},
  volume={201},
  number={2},
  pages={454--462},
  year={2010},
  publisher={Elsevier},
  doi = {10.1016/j.ejor.2009.03.019}
}

@article{besbes2013implications,
  title={On implications of demand censoring in the newsvendor problem},
  author={Besbes, Omar and Muharremoglu, Alp},
  journal={Management Science},
  volume={59},
  number={6},
  pages={1407--1424},
  year={2013},
  publisher={INFORMS}, 
  doi = {https://doi.org/10.1287/mnsc.1120.1654}
}

@article{he2024stochastic,
  title={Stochastic Approximation for Multi-period Simulation Optimization with Streaming Input Data},
  author={He, Linyun and Shanbhag, Uday V and Song, Eunhye},
  journal={ACM Transactions on Modeling and Computer Simulation},
  volume={34},
  number={2},
  pages={1--27},
  year={2024},
  publisher={ACM New York, NY}
}

@inproceedings{barton1993bootstrap,
  title={Uniform and bootstrap resampling of empirical distributions},
  author={Barton, Russell R and Schruben, Lee W},
  booktitle={Proceedings of the 25th conference on Winter simulation},
  pages={503--508},
  year={1993}
}

@article{cheng1997delta,
  title={Sensitivity of computer simulation experiments to errors in input data},
  author={Cheng, Russell CH and Holloand, Wayne},
  journal={Journal of Statistical Computation and Simulation},
  volume={57},
  number={1-4},
  pages={219--241},
  year={1997},
  publisher={Taylor \& Francis}
}

@article{avriel1970value,
  title={The value of information and stochastic programming},
  author={Avriel, Mordecai and Williams, AC},
  journal={Operations Research},
  volume={18},
  number={5},
  pages={947--954},
  year={1970},
  publisher={INFORMS}
}

@article{birge1997stochastic,
  title={State-of-the-art-survey—stochastic programming: Computation and applications},
  author={Birge, John R},
  journal={INFORMS journal on computing},
  volume={9},
  number={2},
  pages={111--133},
  year={1997},
  publisher={INFORMS}
}

@article{hausch1983bounds,
  title={Bounds on the value of information in uncertain decision problems II},
  author={Hausch, DB and Ziemba, WT},
  journal={Stochastics},
  volume={10},
  number={3-4},
  pages={181--217},
  year={1983},
  publisher={Taylor \& Francis}
}

@book{kall1994stochastic,
  title={Stochastic programming},
  author={Kall, Peter and Wallace, Stein W and Kall, Peter},
  volume={5},
  year={1994},
  publisher={Springer}
}

@article{powell2019unified,
  title={A unified framework for stochastic optimization},
  author={Powell, Warren B},
  journal={European journal of operational research},
  volume={275},
  number={3},
  pages={795--821},
  year={2019},
  publisher={Elsevier}
}

@Article{robbins1951stochastic,
  author        = {Robbins, Herbert and Monro, Sutton},
  journal       = {The annals of mathematical statistics},
  title         = {A stochastic approximation method},
  year          = {1951},
  pages         = {400--407},
  comment-lixin = {The first paper of the stochastic gradient descent for stochatic programming approach},
  publisher     = {JSTOR},
  ranking       = {rank5},
}

@article{dantzig1955linear,
  author  = {Dantzig, George B.},
  title   = {Linear programming under uncertainty},
  journal = {Management Science},
  volume  = {1},
  number  = {3--4},
  pages   = {197--206},
  year    = {1955},
  doi     = {10.1287/mnsc.1.3-4.197}
}

@book{birge2011introduction,
  author    = {Birge, John R. and Louveaux, Fran{\c{c}}ois},
  title     = {Introduction to stochastic programming},
  publisher = {Springer},
  address   = {New York},
  edition   = {2},
  year      = {2011},
  doi       = {10.1007/978-1-4614-0237-4}
}

@book{shapiro2021lectures,
  title={Lectures on stochastic programming: modeling and theory},
  author={Shapiro, Alexander and Dentcheva, Darinka and Ruszczynski, Andrzej},
  year={2021},
  publisher={SIAM}
}

@article{kleywegt2002sample,
  title={The sample average approximation method for stochastic discrete optimization},
  author={Kleywegt, Anton J and Shapiro, Alexander and Homem-de-Mello, Tito},
  journal={SIAM Journal on optimization},
  volume={12},
  number={2},
  pages={479--502},
  year={2002},
  publisher={SIAM}
}

@article{kim2014guide,
  title={A guide to sample average approximation},
  author={Kim, Sujin and Pasupathy, Raghu and Henderson, Shane G},
  journal={Handbook of simulation optimization},
  pages={207--243},
  year={2014},
  publisher={Springer}
}

@article{pereira1991multi,
  title={Multi-stage stochastic optimization applied to energy planning},
  author={Pereira, Mario VF and Pinto, Leontina MVG},
  journal={Mathematical programming},
  volume={52},
  number={1},
  pages={359--375},
  year={1991},
  publisher={Springer}
}

@article{seljom2021sample,
  title={Sample average approximation and stability tests applied to energy system design},
  author={Seljom, Pernille and Tomasgard, Asgeir},
  journal={Energy Systems},
  volume={12},
  number={1},
  pages={107--131},
  year={2021},
  publisher={Springer}
}

@article{sodhi2009modeling,
  title={Modeling supply-chain planning under demand uncertainty using stochastic programming: A survey motivated by asset--liability management},
  author={Sodhi, ManMohan S and Tang, Christopher S},
  journal={International Journal of Production Economics},
  volume={121},
  number={2},
  pages={728--738},
  year={2009},
  publisher={Elsevier}
}

@article{dupavcova2003scenario,
  title={Scenario reduction in stochastic programming},
  author={Dupavcova, Jitka and Growe-Kuska, Nicole and Romisch, Werner},
  journal={Mathematical programming},
  volume={95},
  number={3},
  pages={493--511},
  year={2003},
  publisher={Springer}
}

@article{mak1999monte,
  title={Monte Carlo bounding techniques for determining solution quality in stochastic programs},
  author={Mak, Wai-Kei and Morton, David P and Wood, R Kevin},
  journal={Operations research letters},
  volume={24},
  number={1-2},
  pages={47--56},
  year={1999},
  publisher={Elsevier}
}

@article{choi2025generative,
  title={A generative AI-based stochastic optimization framework for optimal site selection and design of power-to-MeOH systems},
  author={Choi, Hyunjun and Kim, Youngkeun and Kim, Jeongdong and Lee, Sungmin and Lee, Man Sig and Kim, Junghwan},
  journal={International Journal of Hydrogen Energy},
  volume={190},
  pages={152114},
  year={2025},
  publisher={Elsevier}
}

@article{bertsimas2023optimization,
  title={Optimization-based scenario reduction for data-driven two-stage stochastic optimization},
  author={Bertsimas, Dimitris and Mundru, Nishanth},
  journal={Operations Research},
  volume={71},
  number={4},
  pages={1343--1361},
  year={2023},
  publisher={INFORMS}
}

@InProceedings{hesong2024introductory,
  author        = {He, Linyun and Song, Eunhye},
  booktitle     = {2024 Winter Simulation Conference (WSC)},
  title         = {Introductory Tutorial: Simulation Optimization Under Input Uncertainty},
  year          = {2024},
  pages         = {1338-1352},
  comment-lixin = {A literature review of the simulation optimization under input data uncertainty},
  doi           = {10.1109/WSC63780.2024.10838862},
}

@article{morton1998stopping,
  title={Stopping rules for a class of sampling-based stochastic programming algorithms},
  author={Morton, David P},
  journal={Operations research},
  volume={46},
  number={5},
  pages={710--718},
  year={1998},
  publisher={INFORMS}
}

@article{pichler2022risk,
  title={Risk-averse stochastic programming: Time consistency and optimal stopping},
  author={Pichler, Alois and Liu, Rui Peng and Shapiro, Alexander},
  journal={Operations Research},
  volume={70},
  number={4},
  pages={2439--2455},
  year={2022},
  publisher={INFORMS}
}

@article{bayraksan2011sequential,
  title={A sequential sampling procedure for stochastic programming},
  author={Bayraksan, G{\"u}zin and Morton, David P},
  journal={Operations Research},
  volume={59},
  number={4},
  pages={898--913},
  year={2011},
  publisher={INFORMS}
}

@article{oh2016characterizing,
  title={Characterizing the structure of optimal stopping policies},
  author={Oh, Sechan and {\"O}zer, {\"O}zalp},
  journal={Production and Operations Management},
  volume={25},
  number={11},
  pages={1820--1838},
  year={2016},
  publisher={SAGE Publications Sage CA: Los Angeles, CA}
}

@article{ciocan2022interpretable,
  title={Interpretable optimal stopping},
  author={Ciocan, Dragos Florin and Mi{\v{s}}i{\'c}, Velibor V},
  journal={Management Science},
  volume={68},
  number={3},
  pages={1616--1638},
  year={2022},
  publisher={INFORMS}
}

@inproceedings{lam2016advanced,
  title={Advanced tutorial: Input uncertainty and robust analysis in stochastic simulation},
  author={Lam, H.},
  booktitle={2016 Winter Simulation Conference (WSC)},
  pages={178--192},
  year={2016},
  organization={IEEE}
}

@inproceedings{wang2023input,
  title={Input Data Collection Versus Simulation: Simultaneous Resource Allocation},
  author={Wang, Y. and Zhou, E.},
  booktitle={2023 Winter Simulation Conference (WSC)},
  pages={3657--3668},
  year={2023},
  organization={IEEE}
}

@article{ungredda2022bayesian,
  title={{B}ayesian optimisation vs. input uncertainty reduction},
  author={Ungredda, J. and Pearce, M. and Branke, J.},
  journal={ACM Transactions on Modeling and Computer Simulation (TOMACS)},
  volume={32},
  number={3},
  pages={1--26},
  year={2022},
  publisher={ACM New York, NY}
}

@inproceedings{lam2015quantifying,
  title={Quantifying uncertainty in sample average approximation},
  author={Lam, H. and Zhou, E.},
  booktitle={2015 Winter Simulation Conference (WSC)},
  pages={3846--3857},
  year={2015},
  organization={IEEE}
}

@inproceedings{song2014advanced,
  title={Advanced tutorial: Input uncertainty quantification},
  author={Song, E. and Nelson, B.L. and Pegden, C.D.},
  booktitle={Proceedings of the Winter Simulation Conference 2014},
  pages={162--176},
  year={2014},
  organization={IEEE}
}

@article{chick2001input,
  title={Input distribution selection for simulation experiments: Accounting for input uncertainty},
  author={Chick, S.E.},
  journal={Operations Research},
  volume={49},
  number={5},
  pages={744--758},
  year={2001},
  publisher={INFORMS}
}

@inproceedings{song2019stochastic,
  title={Stochastic approximation for simulation optimization under input uncertainty with streaming data},
  author={Song, E. and Shanbhag, U.V.},
  booktitle={2019 Winter Simulation Conference (WSC)},
  pages={3597--3608},
  year={2019},
  organization={IEEE}
}

@article{liu2019online,
  title={Online quantification of input model uncertainty by two-layer importance sampling},
  author={Liu, T. and Zhou, E.},
  journal={arXiv preprint arXiv:1912.11172},
  year={2019}
}

@article{aigner2023data,
  title={Data-driven distributionally robust optimization over time},
  author={Aigner, K.M. and B{\"a}rmann, A. and Braun, K. and Liers, F. and Pokutta, S. and Schneider, O. and Sharma, K. and Tschuppik, S.},
  journal={INFORMS Journal on Optimization},
  volume={5},
  number={4},
  pages={376--394},
  year={2023},
  publisher={INFORMS}
}

\end{document}